\RequirePackage{ifthen}
\newboolean{snjnl}
\setboolean{snjnl}{false}

\ifthenelse{\boolean{snjnl}}{%
\documentclass[pdflatex,sn-mathphys-num]{sn-jnl}
}{%
\documentclass[11pt]{article}
}

\usepackage{graphicx}
\usepackage{comment}
\usepackage{multirow}
\usepackage{amsmath,amssymb,amsfonts}
\usepackage{mathrsfs}
\usepackage{xcolor}
\usepackage{textcomp}
\usepackage{booktabs}
\usepackage{listings}
\usepackage{tikz}
\usepackage{pgfplots}
\pgfplotsset{compat=1.18}
\usetikzlibrary{arrows.meta, backgrounds, intersections}
\definecolor{oxfordblue}{RGB}{0,33,71}
\usepackage{accents}
\usepackage{array}
\usepackage{longtable}
\usepackage{pdflscape}
\usepackage{enumitem}
\usepackage{xspace}
\usepackage{bm}
\usepackage{subcaption}

\ifthenelse{\boolean{snjnl}}{%
\usepackage[title]{appendix}
\usepackage{manyfoot}
\DeclareNewFootnote{A}[gobble]
\setlength{\skip\footinsA}{0pt}
\usepackage{algorithm}
\usepackage{algorithmicx}
\usepackage{algpseudocode}
}{%
\usepackage{amsthm}
\usepackage{algorithm}
\usepackage{algorithmicx}
\usepackage{algpseudocode}
\usepackage{hyperref}

}
\usepackage[capitalize,noabbrev]{cleveref}
\algrenewcommand\algorithmicrequire{\textbf{Input:}}
\algrenewcommand\algorithmicensure{\textbf{Output:}}

\ifthenelse{\boolean{snjnl}}{%
\theoremstyle{thmstyleone}
\newtheorem{theorem}{Theorem}
\newtheorem{proposition}[theorem]{Proposition}
\newtheorem{lemma}[theorem]{Lemma}
\newtheorem{corollary}[theorem]{Corollary}
\newtheorem{claim}[theorem]{Claim}

\theoremstyle{thmstyletwo}
\newtheorem{example}{Example}
\newtheorem{remark}{Remark}

\theoremstyle{thmstylethree}
\newtheorem{definition}{Definition}
\newtheorem{assumption}{Assumption}
}{%
\newtheorem{theorem}{Theorem}
\newtheorem{proposition}{Proposition}
\newtheorem{lemma}{Lemma}
\newtheorem{corollary}{Corollary}
\newtheorem{remark}{Remark}

\newtheorem{example}{Example}

}

\ifthenelse{\boolean{snjnl}}{%
\newenvironment{prf}[1][]{\begin{proof}}{\end{proof}}
\newenvironment{prfc}[1][]{\begin{proof}[Proof #1]}{\end{proof}}
\newenvironment{cpf}%
  {\begin{trivlist}\item[]{\em Proof of claim. }}%
  {$\hfill\diamond$\end{trivlist}}
\newenvironment{spf}%
  {\begin{trivlist}\item[]{\em Proof of step. }}%
  {$\hfill\triangleleft$\end{trivlist}}
}{%

  {\begin{trivlist}\item[]{\em Proof of claim. }}%
  {$\hfill\diamond$\end{trivlist}}
  {\begin{trivlist}\item[]{\em Proof of step. }}%
  {$\hfill\triangleleft$\end{trivlist}}
}

\ifthenelse{\boolean{snjnl}}{}{%
\newcommand{\fnm}[1]{#1}
\newcommand{\sur}[1]{#1}

\newcommand{\orgname}[1]{#1}
\newcommand{\orgaddress}[1]{#1}

\newcommand{\city}[1]{#1}

\newcommand{\country}[1]{#1}

\newcommand{\keywords}[1]{\par\smallskip\noindent\textbf{Keywords:} #1}
\providecommand{\subjclass}[2][]{%
  \par\smallskip\noindent\textbf{MSC (#1):} #2%
}
\Crefname{chapter}{Chap.}{Chaps.}
\Crefname{section}{Sect.}{Sects.}
\Crefname{proposition}{Prop.}{Props.}
\Crefname{theorem}{Thm.}{Thms.}
\Crefname{definition}{Defn.}{Defns.}
\Crefname{corollary}{Cor.}{Cors.}
\Crefname{assumption}{Assum.}{Assums.}
\Crefname{lemma}{Lem.}{Lems.}
}

\providecommand{\subjclass}[2][]{%
  \par\addvspace{10pt}{\keywordfont{\bfseries MSC (#1):} #2\par}%
}

\newcommand{\bR}{\mathbb{R}}
\newcommand{\bZ}{\mathbb{Z}}
\newcommand{\bB}{\mathbb{B}}
\newcommand{\bS}{\mathbb{S}}
\newcommand{\cA}{{\mathcal{A}}}

\newcommand{\cE}{{\mathcal{E}}}

\newcommand{\cN}{{\mathcal{N}}}

\newcommand{\cX}{{\mathcal{X}}}

\newcommand{\norm}[1]{{\lVert#1\rVert}}

\newcommand{\relx}[1]{\tilde{#1}}
\newcommand{\MI}{\mathrm{MI}}
\newcommand{\Par}{\mathrm{P}}
\newcommand{\JR}{\mathrm{JR}}
\newcommand{\MP}{\mathrm{MP}}
\newcommand{\BI}{\mathbb{I}}
\newcommand{\idxset}[1]{\{1,\ldots,#1\}}

\DeclareMathOperator{\conv}{conv}

\DeclareMathOperator{\dom}{dom}

\DeclareMathOperator{\rang}{range}
\DeclareMathOperator{\inter}{int}

\DeclareMathOperator{\clos}{cl}

\DeclareMathOperator{\bd}{bd}

\newcolumntype{L}{>{$}l<{$}}
\newcolumntype{C}{>{$}c<{$}}
\newcolumntype{R}{>{$}r<{$}}

\allowdisplaybreaks

\begin{document}

\newcommand{\PaperShortTitle}{Joint-Range Inequalities for QCQPs}
\newcommand{\PaperTitle}{Joint-Range Inequalities for Nonconvex QCQPs}
\newcommand{\PublicationStatusDisclosure}{No part of this manuscript has been published previously or submitted for publication elsewhere.}
\ifthenelse{\boolean{snjnl}}{%
  \title[\PaperShortTitle]{\PaperTitle}
  \artnote{\PublicationStatusDisclosure}
}{%
  \title{\PaperTitle\thanks{\PublicationStatusDisclosure}}
}

\ifthenelse{\boolean{snjnl}}{%
  \author*[1]{\fnm{Liding} \sur{Xu}}\email{lidingxu.ac@gmail.com}
  \author[1,2]{\fnm{Sebastian} \sur{Pokutta}}\email{pokutta@zib.de}

  \affil*[1]{%
    \orgname{Zuse Institute Berlin},
    \orgaddress{\city{Berlin}, \country{Germany}}%
  }
  \affil[2]{%
    \orgname{Technische Universit\"at Berlin},
    \orgaddress{\city{Berlin}, \country{Germany}}%
  }
}{%
  \author{%
    \fnm{Liding} \sur{Xu}%
    \thanks{\orgname{Zuse Institute Berlin},
      \orgaddress{\city{Berlin}, \country{Germany}}.
      E-mail: \texttt{lidingxu.ac@gmail.com}}%
    \and
    \fnm{Sebastian} \sur{Pokutta}%
    \thanks{\orgname{Zuse Institute Berlin} and
      \orgname{Technische Universit\"at Berlin},
      \orgaddress{\city{Berlin}, \country{Germany}}.
      E-mail: \texttt{pokutta@zib.de}}%
  }
  \date{\today}
}

\newcommand{\PaperAbstract}{%
  We study cutting planes for nonconvex quadratically constrained quadratic programs (QCQPs) through a project-then-lift approach inspired by mixed-integer rounding (MIR) inequalities. Given two base valid inequalities for the extended QCQP formulation, we project the associated two-row relaxation into a two-dimensional set and analyze the joint range of quadratic functions in two base inequalities.  For the nonconvex joint range, we give a closed-form convex hull description of the projected set; for the convex joint range, we give its semidefinite representation. This yields a new family of \emph{joint-range inequalities}, which can be lifted back to the extended QCQP formulation. MIR inequalities can handle ``mixed'' terms: continuous variables or fractional linear combinations of integer variables. Similarly, we propose more flexible  \emph{secant mixed-joint-range inequalities}, which better expose and exploit the nonconvex joint range. The proposed approach preserves sparsity, since the support of each lifted inequality is controlled by that of the base inequalities. In preliminary geometric experiments,  the joint-range inequalities yield substantial area reduction of the projected relaxation constructed via reformulation-linearization-technique.
}

\newcommand{\PaperKeywords}{Quadratically constrained quadratic programming, Cutting planes, Projections, S-lemma, Joint range of quadratic functions}
\newcommand{\PaperMSC}{Primary 90C20; Secondary 90C26, 90C57}

\maketitle

\ifthenelse{\boolean{snjnl}}{%
  {\abstractfont
      \abstracthead*{\abstractname}
      \PaperAbstract\par
    }
  \par\addvspace{10pt}{\keywordfont{\bfseries Keywords:} \PaperKeywords\par}
  \subjclass[2020]{\PaperMSC}
}{%
  \begin{abstract}
    \PaperAbstract
  \end{abstract}
  \subjclass[2020]{\PaperMSC}
}


\section{Introduction}\label{sec:intro2}

Quadratically constrained quadratic programs (QCQPs) form an important class of optimization problems at the core of mixed-integer nonlinear programming (MINLP). Their standard form is
\begin{equation}
  \label{eq:qcqp}
  \begin{aligned}
    \min_{x\in\bR^n} \quad & x^\top Q_0 x + q_0^\top x + r_0                       \\
    \text{s.t.}     \quad  & x^\top Q_k x + q_k^\top x \le r_k, \quad \forall k \in \idxset{m}.
  \end{aligned}
\end{equation}
Cutting planes form the computational backbone of modern MINLP solvers \cite{hojny2025scip,bao2009multitermpolyhedral}, but their design is constrained by a persistent trade-off between the tightness of a convexification and the difficulty of optimizing over it~\cite{lee2007mixed}.

The standard extended formulation introduces a variable $X_{ij}$ to represent the product $x_i x_j$ appearing in the objective or in any constraint. Assume that $\cE \subseteq \{\{i,j\}: 1 \le i \le j \le n\}$ indexes these products. Let $\bS^n$ denote the space of $n \times n$ real symmetric matrices and let $\bS^n_{\cE}\doteq \{X \in \bS^n: \forall \{i,j\} \notin \cE,\; X_{ij} = 0\}$ denote the sparse subspace of matrices with support contained in $\cE$.
The extended formulation reads
\begin{equation}
  \label{eq:lifted-formulation-intro}
  \begin{aligned}
    \min_{x \in \bR^n, X \in \bS^{n}_{\cE}} \quad & \langle Q_0, X \rangle + q_0^\top x + r_0        &                                  \\
    \text{s.t.}                    \quad      & \langle Q_k, X \rangle + q_k^\top x \le r_k, \quad & \forall \, k \in \idxset{m}, \\
& X_{ij} = x_i x_j, \quad &\forall\, \{i,j\} \in \cE.
  \end{aligned}
\end{equation}
Since all $Q_k \in \bS^n_{\cE}$, the bilinear constraint  $X_{ij} = x_i x_j$  enforces the sparse rank-one lifting constraint $X_{\cE}=(xx^\top)_{\cE}$. Without loss of generality, we can set $X \in \bS^{n}_{\cE}$ as well.
For this nonconvex formulation, the bilinear constraints are commonly relaxed by linear Reformulation-Linearization-Technique (RLT) inequalities~\cite{mccormick1976computability,qualizza2011linear,sherali1999reformulation} for bounded variables. A subsequent line of work strengthens RLT-type relaxations by deriving sharper
convex hull descriptions or strong inequalities for bilinear and quadratic
structures~\cite{anstreicher2021convexhullboundedproducts,gu2022liftingconvexbipartite,davarnia2017simultaneousconvexification,zhu2026axisalignedrelaxations,Locatelli2014}.
A complementary line of work treats $X$ as  a dense matrix and imposes positive semidefinite (PSD) constraints \cite{anstreicher2009sdp,bao2011semidefinite}, outer-approximated by eigenvector inequalities ~\cite{qualizza2011linear,sherali2002enhancing} for strengthening linear programming (LP) relaxations. Some approaches exploit eigenvalue reformulations of the quadratic functions~\cite{billionnet2012extending,dong2018compact,elloumi2019global,saxena2010convex}. However, eigenvector inequalities and eigenvalue reformulations typically do not preserve sparsity \cite{dey2022cutting}: the resulting inequalities can be dense.

Intersection cut methods for QCQPs are closely related to earlier developments for mixed-integer linear programming (MILP)~\cite{balas1971intersection} and nonlinear optimization~\cite{tuy1964concave}. A common \emph{a priori} strategy first selects an $S$-free set, such as a bilinear-free, outer-product-free, or quadratic-free set~\cite{bienstock2020outer,fischetti2020branch,munoz2022maximal,Munoz2025}, and then derives a cut from the simplicial cone associated with the tableau of an LP relaxation \cite{Chmiela2025}. Although these free sets could admit low-dimensional descriptions, the simplicial cone still lives in the original variable space, so the generated cuts can be dense \cite{XuLiberti2025}. This suggests seeking a framework in which a low-dimensional and sparse simplicial cone is constructed first, whereas the cut-generation process is subsequently coordinated with that cone.
Related examples include SDP and RLT representations \cite{anstreicher2010computable}  for the convex hulls of low-dimensional quadratic functions over simple domains.

The derivation of general-purpose, sparsity-preserving, and strong valid inequalities for QCQPs that can be effectively integrated into MINLP solvers remains a significant challenge.
We are motivated by the fact that mixed-integer rounding (MIR) inequalities can accelerate MILP solving~\cite{nemhauser1988integer,nemhauser1990recursive} through an efficient and sparsity-preserving heuristic \cite{marchand2001aggregation}. MIR inequalities show that one can begin with a small number of base valid inequalities, derive strong cuts by adaptive rounding rather than from a fixed disjunction, and still maintain a useful balance between strength and sparsity. At the elementary-closure level, MIR and split/disjunctive inequalities are identical \cite{balas1998disjunctive,cornuejols2001elementary}, so both families of inequalities are of the same strength and $\mathcal{NP}$-hard to separate. MIR inequalities rely solely on the integrality of a single aggregated variable, and the generalization to two-row intersection cuts~\cite{Andersen07} leverages a richer nonconvex geometry formed from two aggregated integer variables.

 We use this perspective to ask whether QCQP cuts can be generated from an equally small projected description. We consider a general project-then-lift approach to generate valid inequalities for a nonconvex set $\cX \subset \bR^n$: one identifies base valid inequalities, forms an affine map $\cA: \bR^n \to \bR^p$ that projects them into a simplicial cone $K \supseteq \cA(\cX)$ in the image space, identifies a set $S \supseteq \cA(\cX)$ capturing the relevant nonconvex geometry, derives valid inequalities for the closed convex hull $\clos(\conv(S \cap K))$, and then lifts them back to the original space. When the complement set $S^c$ is convex and the apex of the simplicial cone lies in the interior of $S^c$, the projected inequalities become intersection cuts in the image space.

  For QCQPs, we exploit the structure of the constraint system with the sparse lifting
  $X_{\cE}=(xx^\top)_{\cE}$ instead of considering the lifting alone.
  Suppose that the extended formulation contains the following two base valid inequalities, which together define a two-row relaxation of
    \eqref{eq:lifted-formulation-intro}:
\begin{equation}
  \label{eq:base-cut}
  \begin{aligned}
     & \phi_1 + \langle \Theta_1, X \rangle + \theta_1^\top x \ge 0, \\
     & \phi_2 + \langle \Theta_2, X \rangle + \theta_2^\top x \ge 0,
  \end{aligned}
\end{equation}
where $\Theta_1,\Theta_2\in\bS^n_{\cE}$ in accordance with the support convention above (e.g., both inequalities can be obtained via aggregations of constraints in \eqref{eq:lifted-formulation-intro}).
In the context of project-then-lift, we let the affine map $\cA$ be $ \cA_\JR(X,x)\doteq( \langle \Theta_1, X \rangle + \theta_1^\top x, \langle \Theta_2, X \rangle + \theta_2^\top x)$ and denote $y \doteq \cA_\JR(X,x)$. Consider the associated quadratic map $F=(f_1,f_2):\bR^n \to \bR^2$, with $f_i(x) \doteq x^\top \Theta_i x + \theta_i^\top x$. When $X_{\cE} = (xx^\top)_{\cE}$, we have $y_i = f_i(x)$, i.e., $y \in F(\bR^n)$. The simplicial cone $K$ becomes a two-dimensional cone $K_\JR \doteq  \{ y \in \bR^2: \forall i \in \{1,2\},\; \phi_i+ y_i \ge 0\}$, which is derived from the affine map applied on the base inequalities \eqref{eq:base-cut}. Then, we apply the project-then-lift approach, handling the geometry where it is simplest.

The complete characterization of $F(\bR^n)$~\cite{flores2016characterizing,flores2021}
distinguishes between convex and nonconvex cases; the convex case is closely related
to the hidden convexity underlying the S-lemma~\cite{dines1941mapping,polik2007survey,
polyak1998convexity,yakubovich1971s}. This hidden convexity has been widely exploited
to tackle QCQPs~\cite{blekherman2024aggregations,beck2006strong,burer2015trust,
dey2022obtaining,wang2022tightness,xia2016s,eichfelder2026convexlikeness}. In particular, the convex hull of the
intersection of two open quadratic sets in $x$-space can be obtained through aggregation
and SDP lifting~\cite{yildiran2009convex}. Related conic outer approximations, together
with exactness conditions, have been studied for intersections of one conic quadratic set
and one quadratic set~\cite{burer2017second,modaresi2017convex}. In this work, we expose nonconvex joint ranges and use their
closed-form descriptions to generate cuts for the extended QCQP
formulation \eqref{eq:lifted-formulation-intro} in $(X,x)$-space.

Our main contributions follow the project-then-lift pipeline. First, for nonconvex
$F(\bR^n)$, we characterize a closed-form description of the (closed) convex hull
$\clos(\conv(F(\bR^n)\cap K_\JR))$; for convex $F(\bR^n)$, we provide an SDP
representation. Second, we derive \emph{joint-range inequalities} that are valid for
the convex hull and lift them back through the affine map. This project-then-lift
approach preserves sparsity in both the nonconvex and convex SDP cases, because the
support of each lifted joint-range inequality is controlled by the support of the
selected base inequalities. Third, we introduce \emph{secant mixed-joint-range
inequalities}, which are analogous to MIR inequalities in their treatment of ``problematic''
fractional linear combinations of integer variables as a  ``mixed'' continuous  variable and are designed to expose nonconvex geometry for cut
generation~\cite{nemhauser1988integer,nemhauser1990recursive}.

The paper is organized as follows. \Cref{sec:milp} reviews MIR inequalities from a project-then-lift perspective that motivates our construction.  \Cref{sec.intersect} proposes to combine intersection cuts with the project-then-lift approach. \Cref{sec:parabolic} studies the two-dimensional convex hull $\clos(\conv(F(\bR^n)\cap K_\JR))$ and the resulting joint-range inequalities. \Cref{sec:MIR} develops secant mixed-joint-range inequalities. \Cref{sec:experimental} reports preliminary geometric experiments comparing the projected RLT relaxation with the convex hull.

\section{Mixed-Integer Rounding Prototypes}
\label{sec:milp}
We review mixed-integer rounding inequalities as prototypes and show that they can be derived via the project-then-lift approach.
We consider the mixed-integer set
\begin{equation}
  \label{eq.milp}
  \cX_{\MI} \doteq (\bR^{n_N} \times \bZ^{n_I}) \cap P,
\end{equation}
where $\bZ$ is the set of integers and $P$ is a polyhedron.  Let $N \doteq \{1,\ldots,n_N\}$ and $I \doteq \{n_N+1,\ldots,n_N+n_I\}$, and write $x_N=(x_i)_{i\in N}$ and $x_I=(x_i)_{i\in I}$ for the continuous and integer components.

We also consider the simple mixed-integer set defined as
\begin{equation}
  \label{eq.mi-set}
  Y_{\MI} \doteq \{(y_1,y_2) \in \bZ \times \bR_+ : y_1 \le \pi_0 + y_2 \}.
\end{equation}
Equivalently, $Y_{\MI}=S_{\MI}\cap K_{\MI}$ with $S_{\MI} \doteq \bZ\times\bR$ and
\begin{equation}
  \label{eq.cone-mir}
  K_{\MI} \doteq \{(y_1,y_2): y_1-y_2\le \pi_0,\ y_2\ge 0\}.
\end{equation}
We call $\conv(Y_{\MI})$ the simple mixed-integer hull.

\begin{lemma}[\cite{nemhauser1988integer}]\label{lem:MIR}
  Let \(f_0 \doteq \pi_0-\lfloor \pi_0\rfloor\).  The simple MIR inequality
  \begin{equation}\label{eq:MIR}
    y_1 \le \lfloor \pi_0\rfloor + \frac{y_2}{1-f_0}
  \end{equation}
  is valid for \(Y_{\MI}\).  Together with the inequalities defining $K_{\MI}$, it describes the simple mixed-integer hull $\conv(Y_{\MI})$; see \Cref{fig:mir}.
\end{lemma}

\begin{figure}[!htbp]
  \centering
  \begin{tikzpicture}[auto,scale=1,
    background rectangle/.style={draw=oxfordblue!20, fill=white, rounded corners}, 
    show background rectangle, 
    inner frame sep=10pt
]
\begin{axis}[
    axis lines = middle,
    axis on top=true,
    xlabel = {$y_1$},
    ylabel = {$y_2$},
    xmin = 0, xmax = 2.4,
    ymin = -0.2, ymax = 1.2,
    xtick = {0, 0.5, 1, 1.5, 2},
    ytick = {0, 0.5, 1},
    grid = both,
    minor tick num = 4,
    grid style = {line width=.1pt, draw=gray!20},
    major grid style = {line width=.2pt, draw=gray!50},
    width = 0.6\linewidth, height = 0.48\linewidth,
    clip = true
]

\addplot[name path=A, domain=0:2.5, draw=none] {1.2};
\addplot[name path=B, domain=0:1.5, draw=none] {0};
\addplot[name path=C, domain=1.5:2.5, draw=none] {1.2*(x-1.5)};

\fill[gray!40] (axis cs:0,0) -- (axis cs:0,1.2) -- (axis cs:2.5,1.2) -- (axis cs:1.5,0) -- cycle;

\draw[thick] (axis cs:0, -0.2) -- (axis cs:0, 1.2);
\draw[thick] (axis cs:1, -0.2) -- (axis cs:1, 1.2);
\draw[thick] (axis cs:2, -0.2) -- (axis cs:2, 1.2);

\addplot[domain=0.3:2.4, color=red!70!black, ultra thick] {0.6*(x-1)};

\end{axis}
\end{tikzpicture}
  \caption{The simple mixed-integer set $Y_{\MI} = \{(y_1,y_2)\in\bZ\times\bR_+ : y_1 \le \pi_0 + y_2\}$ and the simple mixed-integer hull $\conv(Y_{\MI})$ (shaded), whose upper boundary is defined by the MIR inequality \eqref{eq:MIR}.}
  \label{fig:mir}
\end{figure}
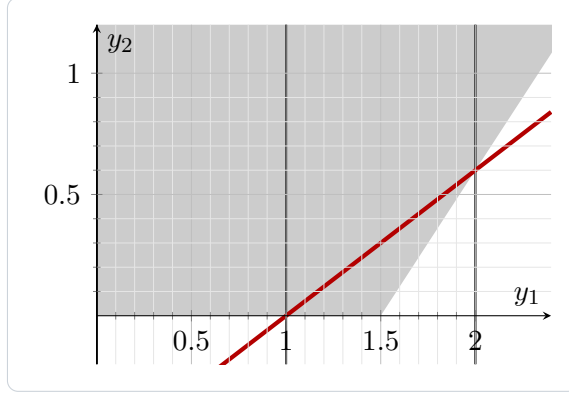

\begin{lemma}[\cite{nemhauser1988integer}]
  \label{lem:alternative-mir}
  Assume that a valid inequality for $P$ can be written as
  \begin{equation}
    \label{eq:alternative-mir-form}
    \pi_I^\top x_I - (\mu_0+\mu^\top x) \le \pi_0,
  \end{equation}
  where $\pi_I\in\bZ^{n_I}$ and $\mu_0+\mu^\top x\ge 0$ is valid for $P$.  With $f_0 \doteq \pi_0-\lfloor\pi_0\rfloor$, the MIR inequality
  \begin{equation}
    \label{eq:alternative-mir}
    \pi_I^\top x_I
    - \frac{1}{1-f_0}(\mu_0+\mu^\top x)
    \le \lfloor\pi_0\rfloor
  \end{equation}
  is valid for $\cX_{\MI}$ in \eqref{eq.milp}.
\end{lemma}

\begin{remark}[MIR inequalities as lifting of projected  inequalities]
  \label{rem:mir-projected}
  Define the affine map
  \begin{equation}
    \label{eq:mir-affine-map}
    \cA_{\MI}:x\mapsto (y_1,y_2)
    \doteq (\pi_I^\top x_I,\; \mu_0+\mu^\top x).
  \end{equation}
  For $x\in\cX_{\MI}$, the value $y_1$ is integral. For $x\in P$, we have the continuous ``mixed'' term $y_2\ge 0$, and \eqref{eq:alternative-mir-form} maps to $y_1-y_2\le \pi_0$. Hence $\cA_{\MI}(P)$ is included in the cone $K_{\MI}$ defined by \eqref{eq.cone-mir}, and the simple MIR inequality \eqref{eq:MIR} lifts exactly to MIR inequality~\eqref{eq:alternative-mir}.
\end{remark}

\begin{remark}[MIR inequalities versus Chv\'atal inequalities]
  \label{rem:mir-vs-chvatal}
  Setting the ``mixed'' term $y_2$ to zero gives the simple rounding inequality
  \begin{equation}
    \label{eq:rounding}
    y_1 \le \lfloor \pi_0 \rfloor
  \end{equation}
  valid for the simple integer set
  \begin{equation}
    \label{eq:integer-set}
    \{y_1\in\bZ: y_1 \le \pi_0\}.
  \end{equation}
  Under the lifting $y_1=\pi_I^\top x_I$, this becomes the Chv\'atal inequality \cite{chvatal1973edmonds}:
  \begin{equation}
    \label{eq:chvatal-inequality}
    \pi_I^\top x_I \le \lfloor \pi_0\rfloor.
  \end{equation}
  MIR is more flexible because the ``mixed'' continuous variable $y_2$ can relax the problematic  fractional term  $ \mu_0+\mu^\top x$ in \eqref{eq:mir-affine-map} of the aggregation.
\end{remark}

The broad applicability of MIR inequalities to general MILP problems rests on the aggregation-based MIR heuristic \cite{marchand2001aggregation}.

\section{Project-then-Lift and Intersection Cuts}
\label{sec.intersect}
Intersection cuts \cite{balas1971intersection,tuy1964concave} provide the second useful interpretation of MIR inequalities and help construct ``mixed'' inequalities for QCQPs.  Let $S\subseteq\bR^p$.  A closed convex set $C$ is \emph{$S$-free} if $\inter(C)\cap S=\varnothing$.  Let $K$ be a simplicial cone with apex $y'$ and hyperplane/ray representations:
\begin{equation}
  \label{eq:intersection-cone}
  K = \{y\in\bR^p: B(y-y')\le 0\}
  = \bigl\{y' + \textstyle\sum_{j=1}^{p}\eta_j r_j:\eta\in\bR^p_+\bigr\},
\end{equation}
where $B$ is invertible, $r_j$ is the $j$-th column of $-B^{-1}$, and $\eta_j \doteq -B_j(y-y')$.  If $y'\in\inter(C)$, the step length along each ray is
\begin{equation}
  \label{eq:intersection-step-length}
  \eta^*_j \doteq \sup\{\eta_j\ge 0: y' + \eta_j r_j \in C\}.
\end{equation}
We allow $\eta^*_j = \infty$ (an infinite intersection point) and define $1/\infty = 0$.  The standard intersection cut for $K\setminus\inter(C)$ is
\begin{equation}
  \label{eq:intersection-cut}
  \sum_{j=1}^{p}\frac{1}{\eta^*_j}B_j(y-y') \le -1,
\end{equation}
equivalently $\sum_{j=1}^{p}\eta_j/\eta^*_j\ge 1$.

We describe the difference between project-then-lift intersection cuts and full-space intersection cuts. The former are derived from cuts in a low-dimensional projected space, whereas the latter are derived from cuts in the full-dimensional space.

Let $\cA:\bR^n\to\bR^p$, $x\mapsto Ax+a$, be an affine map. As the MIR discussion shows, inequalities can be derived from an outer approximation of the projected feasible region $\cA(\cX_{\MI})$ of a MILP. The same strategy can be used to derive intersection cuts for any nonconvex set $\cX \subseteq \bR^n$. The following theorem describes a project-then-lift intersection cut (for short, lifted intersection cut).

\begin{theorem}
  \label{thm:lifted-intersection-cut}
  Let $\cX\subseteq\bR^n$, let $S\subseteq\bR^p$ satisfy $\cA(\cX)\subseteq S$, and let $C\subseteq\bR^p$ be a full-dimensional closed convex $S$-free set.  Let $K\supseteq\cA(\cX)$ be a simplicial cone with apex $y'$ and representation \eqref{eq:intersection-cone}.  If $y'\in\inter(C)$, then
  \begin{equation}
    \label{eq:lifted-intersection-cut}
    \sum_{j=1}^{p}\frac{1}{\eta^*_j}B_j\bigl(\cA(x)-y'\bigr) \le -1
  \end{equation}
  is valid for $\cX$.
\end{theorem}
\begin{proof}
  Since $\cA(\cX)\subseteq S$ and $C$ is $S$-free, $\cA(\cX)\cap\inter(C)=\varnothing$.  Together with $\cA(\cX)\subseteq K$, this implies $\cA(\cX)\subseteq K\setminus\inter(C)$.  The intersection cut \eqref{eq:intersection-cut} is valid for $\conv(K\setminus\inter(C))$, hence for $\cA(\cX)$.  Substituting the image-space point $y=\cA(x)$ gives \eqref{eq:lifted-intersection-cut}.
\end{proof}

For comparison, one can lift the forbidden set rather than the cut.  The preimage
\[
  \cA^{-1}(C) = \{x\in\bR^n: Ax+a\in C\}
\]
is $\cX$-free under the following standard condition.

\begin{theorem}
  \label{thm:preimage-free}
  Let $\cX\subseteq\bR^n$, and let $\cA:\bR^n\to\bR^p$, $x\mapsto Ax+a$, be a surjective affine map.  If $C\subseteq\bR^p$ is a closed, convex, $\cA(\cX)$-free set, then $\cA^{-1}(C)$ is a closed, convex, $\cX$-free set.
\end{theorem}
\begin{proof}
  Closedness follows from continuity  of $\cA$, and convexity holds because $\cA$ is affine.  If $x\in\inter(\cA^{-1}(C))$, then an open neighborhood $U$ of $x$ satisfies $U\subseteq\cA^{-1}(C)$.  Because surjectivity makes an affine map open, $\cA(U)$ is an open neighborhood of $\cA(x)$ contained in $C$, so $\cA(x)\in\inter(C)$.  Thus any $x\in \cX\cap\inter(\cA^{-1}(C))$ would imply $\cA(x)\in\cA(\cX)\cap\inter(C)$, a contradiction.
\end{proof}
Note that this lifting is similar to the lifting of a lattice-free set to a mixed-integer-free set; see \cite{conforti2011corner}.

In this case, we usually seek a corner polyhedron in the original space, rather than in the projected space.
Let $P_{\mathrm{LP}}\supseteq \cX$ be an outer approximation corresponding to the linear relaxation.  Suppose $x'$ is a basic solution corresponding to a vertex of $P_{\mathrm{LP}}$, and let $B_{\mathrm{tab}}$ be the nonsingular matrix formed by the active tableau rows at $x'$.  Let $ K_{\mathrm{tab}}$ be the associated tableau cone:
\[
  K_{\mathrm{tab}} = \{x\in\bR^n: B_{\mathrm{tab}}(x-x')\le 0\}
  = \bigl\{x' + \textstyle\sum_{i=1}^{n}\xi_i r'_i:\xi\in\bR^n_+\bigr\},
\]
where $r'_i$ is the $i$-th column of $-(B_{\mathrm{tab}})^{-1}$ and $\xi_i \doteq -(B_{\mathrm{tab}})_i(x-x')$.  If, in addition, $x'\in\inter(\cA^{-1}(C))$, then the usual full-space intersection cut is
\begin{equation}
  \label{eq:fullspace-intersection-cut}
  \sum_{i=1}^{n}\frac{1}{\xi_i^*}(B_{\mathrm{tab}})_i(x-x') \le -1,
\end{equation}
where
\[
  \xi_i^* \doteq \sup\{\xi_i \ge 0: x' + \xi_i r'_i\in \cA^{-1}(C)\}.
\]

This cut is valid for $\cX$ because $\cX\subseteq P_{\mathrm{LP}}\subseteq K_{\mathrm{tab}}$ and $\cA^{-1}(C)$ is $\cX$-free.  It combines $n$ tableau rows.  If each row of $B_{\mathrm{tab}}$ and of the product $BA$ (where $B$ is the matrix from~\eqref{eq:intersection-cone}) has at most $k$ nonzeros, \eqref{eq:fullspace-intersection-cut} can have up to $nk$ nonzeros, whereas the projected lift \eqref{eq:lifted-intersection-cut} combines only $p$ rows of $BA$ and has at most $pk$ nonzeros. When $p\ll n$ and the row sparsities of $B_{\mathrm{tab}}$ and $BA$ are comparable, this suggests that lifted cuts can be substantially sparser than full-space intersection cuts. In addition, the construction of lifted intersection cuts does not require a simplex tableau from the linear relaxation.

\begin{remark}[MIR  inequalities as lifted intersection cuts]
  \label{rem:mir-as-lifted-intersection-cut}
  The simple MIR inequality, and therefore the lifted MIR inequalities above, can also be viewed as a lifted intersection cut.  Use the affine image in \eqref{eq:mir-affine-map} and introduce a slack $s\ge 0$ through $y_1+s=\pi_0+y_2$.  Let $\eta=(\eta_1,\eta_2)\doteq(y_2,s)$, so
  \[
    y_1 = \pi_0+\eta_1-\eta_2.
  \]
  For $f_0\in(0,1)$, the interval $C_{\BI}\doteq[\lfloor\pi_0\rfloor,\lfloor\pi_0\rfloor+1]$ is $\bZ$-free and contains $\pi_0$ in its interior.  Its preimage in slack space is
  \[
    C_{\MI} \doteq \{\eta\in\bR^2:\pi_0+\eta_1-\eta_2\in C_{\BI}\}.
  \]
  The two coordinate step lengths are
  \begin{align*}
    \eta_1^* & = \sup\{\eta_1\ge 0:(\eta_1,0)\in C_{\MI} \}
    = \lfloor\pi_0\rfloor+1-\pi_0 = 1-f_0,                   \\
    \eta_2^* & = \sup\{\eta_2\ge 0:(0,\eta_2)\in C_{\MI}  \}
    = \pi_0-\lfloor\pi_0\rfloor = f_0.
  \end{align*}
  After substituting $\eta_1=y_2$ and $\eta_2=s=\pi_0+y_2-y_1$, the intersection cut
  \[
    \frac{\eta_1}{1-f_0}+\frac{\eta_2}{f_0}\ge 1
  \]
  becomes a simple MIR inequality:
  \[
    \frac{y_2}{1-f_0}
    + \frac{\pi_0+y_2-y_1}{f_0}
    \ge 1
    \quad\Longleftrightarrow\quad
    y_1 \le \lfloor\pi_0\rfloor+\frac{y_2}{1-f_0}.
  \]
\end{remark}

\section{Projection of QCQPs and Joint-Range Inequalities}\label{sec:parabolic}

In this section, we apply the  project-then-lift approach to QCQP: we start from two base inequalities, map them into a two-dimensional image space, characterize the resulting projected set, and then lift its valid inequalities back to the extended QCQP formulation. The construction is similar to that for MIR inequalities, but the key projected geometry now comes from the joint range of two quadratic functions.
Like Chvátal inequalities \eqref{eq:chvatal-inequality} and simple rounding inequalities \eqref{eq:rounding}, joint-range inequalities do not involve ``mixed'' terms (\Cref{rem:mir-vs-chvatal}); the mixed extension is treated in the next section.

Suppose that the projected construction starts from the two base inequalities~\eqref{eq:base-cut}. They define a two-dimensional image space through the affine map
\begin{equation}
  \label{eq:affine-map}
  \begin{aligned}
    \cA_{\JR}\colon \bS^n \times \bR^n & \to \bR^2,               \\
    (X,x)                                   & \mapsto (y_1,y_2) =
    (\langle \Theta_1, X \rangle + \theta_1^\top x,\;
    \langle \Theta_2, X \rangle + \theta_2^\top x).
  \end{aligned}
\end{equation}
Recall that we denote $y = \cA_{\JR}(X,x)$.
The associated quadratic map is
\begin{equation}
  \label{eq.quadmap}
  F:\bR^n \to \bR^2, \quad x \mapsto (x^\top \Theta_1 x + \theta_1^\top x,\; x^\top \Theta_2 x + \theta_2^\top x).
\end{equation}
Since $\Theta_1, \Theta_2 \in \bS^n_{\cE}$, given any (QCQP-feasible) pair $(X,x)$ satisfying the sparse lifting $X_{\cE}=(xx^\top)_{\cE}$, the aggregated variable $y=\cA_{\JR}(X,x)$ lies in the joint range $F(\bR^n)$.

The joint range $F(\bR^n)$ has been completely characterized in \cite{flores2016characterizing}: it is either convex or takes an explicit nonconvex form. For our purposes, we restate the relevant characterization in a constructive form in the appendix and summarize the resulting procedure as Algorithm~\ref{alg:joint-range}.

\begin{algorithm}[H]
  \caption{Constructive Classification of the Joint Range $F(\bR^n)$}
  \label{alg:joint-range}
  \begin{algorithmic}[1]
    \Require Symmetric matrices $\Theta_1, \Theta_2 \in \bR^{n \times n}$ and vectors $\theta_1, \theta_2 \in \bR^n$
    \Ensure Convexity of $F(\bR^n)$ and its geometric shape if nonconvex
    \If{$\Theta_1=0$ and $\Theta_2=0$}
    \State \Return \textsc{Convex}
    \EndIf
    \If{$\Theta_1$ and $\Theta_2$ are linearly independent}
    \State \Return \textsc{Convex}
    \EndIf
    \State Find normalized direction $d = (d_1, d_2) \neq 0$ such that $d_2 \Theta_1 = d_1 \Theta_2$
    \State $\cN_0 \gets \ker \Theta_1 \cap \ker \Theta_2$
    \If{$\{\theta_1,\theta_2\} \not\subseteq \cN_0^\perp$}
    \State \Return \textsc{Convex}
    \EndIf
    \For{candidate direction $\hat{d} \in \{d, -d\}$}
    \State $\bar Q_{\hat d} \gets \hat{d}_1 \Theta_1 + \hat{d}_2 \Theta_2$
    \State Compute the inertia $(m_+,m_-,m_0)$ of $\bar Q_{\hat d}$ on $\cN_0^\perp$
    \If{$m_-\neq 1$}
    \State \textbf{continue} to the next candidate direction
    \EndIf
    \State Apply the canonical reduction in Appendix~\ref{sec:app-canonical-transformation} to obtain $m_+$, $\Delta$, $t_1$, $t_2$, and the coordinate offset $c=(c_d,c_{d_\perp})^\top$
    \If{$\Delta > 0$ \textbf{or} $t_2 = 0$}
    \State \textbf{continue} to the next candidate direction
    \EndIf
    \State Set $\hat d_\perp=(-\hat d_2,\hat d_1)$ and $D=[\,\hat d\ \hat d_\perp\,]$, and set the canonical shape $\relx{S}_{\JR}$ as follows:
    \Statex \hspace{\algorithmicindent}$\relx{S}_{\JR} = \{(\relx{y}_1,\relx{y}_2) : \relx{y}_1 = \relx{y}_2^2/\Delta\}$ if $\Delta < 0$, $m_+ = 0$
    \Statex \hspace{\algorithmicindent}$\relx{S}_{\JR} = \{(\relx{y}_1,\relx{y}_2) : \relx{y}_1 \ge \relx{y}_2^2/\Delta\}$ if $\Delta < 0$, $m_+ \ge 1$
    \Statex \hspace{\algorithmicindent}$\relx{S}_{\JR} = \bR^2 \setminus \{(\relx{y}_1,0) : \relx{y}_1 \neq 0\}$ if $\Delta = 0$, $m_+ = 1$
    \Statex \hspace{\algorithmicindent}$\relx{S}_{\JR} = \bR^2 \setminus \{(\relx{y}_1,0) : \relx{y}_1 < 0\}$ if $\Delta = 0$, $m_+ \ge 2$
    \State \Return \textsc{Nonconvex}, with $F(\bR^n)=D(\relx{S}_{\JR}+c)$
    \EndFor
    \State \Return \textsc{Convex}
  \end{algorithmic}
\end{algorithm}

The correctness of Algorithm~\ref{alg:joint-range} follows from \Cref{thm:classification} in \Cref{appendix.a}. In the nonconvex case, the classification gives explicit shapes; see \Cref{ex:parabolic} for a concrete case.

Under the affine map $\cA_{\JR}$,  the two base inequalities \eqref{eq:base-cut} define the simplicial cone
\begin{equation}
  \label{eq:KJR}
  K_{\JR} = \{(y_1,y_2)\in\bR^2 : \phi_1  + y_1 \ge 0,\; \phi_2 + y_2 \ge 0\}.
\end{equation}
We consider
\begin{equation}
  \label{eq:projected-joint-range-set}
  S_{\JR} \doteq F(\bR^n),\quad Y_{\JR} \doteq S_{\JR} \cap K_{\JR},\quad H_{\JR} \doteq \clos(\conv(Y_{\JR})).
\end{equation}
We call $Y_{\JR}$ the joint-range set and $H_{\JR}$ the joint-range hull. We call the valid inequalities of $H_{\JR}$  \emph{joint-range inequalities}, which can be lifted back to QCQP.

\subsection{Closed Convex Hull for the Nonconvex  Joint Range}
\label{sec:conv-hull-nonconvex}

For nonconvex $S_{\JR}$, we construct the joint-range hull $H_{\JR}$ and derive its valid inequalities.
When $\Delta = 0$, the canonical shape $\relx{S}_{\JR}$ is a plane missing a line or a ray, so the joint-range hull $H_{\JR}=K_{\JR}$ and no nontrivial valid inequalities arise beyond those defining $K_{\JR}$.

When $\Delta < 0$, the joint range $S_{\JR} = D(\relx{S}_{\JR}+c)$ takes one of two nontrivial nonconvex forms (the affine map $D$ preserves the shape of $\relx{S}_{\JR}$):
\begin{itemize}
  \item $\relx{S}_{\JR}=\{(\relx{y}_1,\relx{y}_2): \relx{y}_1 = \relx{y}_2^2/\Delta\}$, a parabola (boundary case);
  \item $\relx{S}_{\JR}=\{(\relx{y}_1,\relx{y}_2): \relx{y}_1 \ge \relx{y}_2^2/\Delta \}$, a solid parabolic region (solid case).
\end{itemize}
Define the bowl function
\begin{equation}
  \label{eq:q-def}
  q(\relx{y}) \;\doteq\; \relx{y}_1 - \frac{\relx{y}_2^2}{\Delta},
\end{equation}
and the parabolic bowl $C_{\Par} \doteq \{\relx{y} \in \bR^2 : q(\relx{y}) \le 0\}$ with boundary $\bd(C_{\Par}) = \{q(\relx{y}) = 0\}$. Then $\relx{S}_{\JR}$ is either the boundary $\bd(C_{\Par})$ or the reverse-convex set $\clos(C_{\Par}^c)$.

To pass between canonical and original coordinates, define the affine map
\begin{equation}
  T:\bR^2 \to \bR^2, \qquad \relx{y} \mapsto T(\relx{y}) \doteq D(\relx{y} + c).
\end{equation}
Since $S_{\JR} = T(\relx{S}_{\JR})$, we work in canonical coordinates $\relx{y} = (\relx{y}_1, \relx{y}_2)$ where the parabola takes its standard form. The inverse is
\begin{equation}
  \label{eq:coord-transform}
  \relx{y} = T^{-1}(y) = D^{-1}y - c,
\end{equation}
under which $y \in S_{\JR}$ if and only if $\relx{y} \in \relx{S}_{\JR}$. The simplicial cone $K_{\JR}$ transforms to the simplicial cone
\begin{equation}
  \label{eq:transformed-cone}
  \relx{K}_{\JR} \doteq T^{-1}(K_{\JR}) = \{\relx{y} \in \bR^2 : \relx{\phi}_1 + \relx{r}_1^\top \relx{y} \ge 0,\; \relx{\phi}_2 + \relx{r}_2^\top \relx{y} \ge 0\},
\end{equation}
where $\relx{r}_i \doteq D^\top e_i$ and
$\relx{\phi}_i \doteq \phi_i + e_i^\top Dc$. Since $T$ is an
affine bijection, the joint-range set $Y_{\JR} = K_{\JR} \cap S_{\JR}$ maps to the corresponding joint-range set
$\relx{Y}_{\JR} = \relx{K}_{\JR} \cap \relx{S}_{\JR}$, and it suffices to characterize the corresponding joint-range hull
$\relx{H}_{\JR} \doteq \clos(\conv(\relx{Y}_{\JR}))$ in canonical coordinates due to the following equivalence:
\[
  H_{\JR}
  = \clos\!\bigl(\conv(Y_{\JR})\bigr)
  = T(\clos(\conv(\relx{Y}_{\JR})))
  = T(\relx{H}_{\JR}).
\]

Recall that $D=[\hat d,\hat d_\perp]$, where $\hat d$ is normalized and
$\hat d_\perp=(-\hat d_2,\hat d_1)$; hence the columns of $D$ are orthonormal
and $D^{-1}=D^\top$. Let
$\relx{\phi}\doteq(\relx{\phi}_1,\relx{\phi}_2)^\top$. The apex
$\relx{y}'$ of the simplicial cone $\relx{K}_{\JR}$ is the unique point where both inequalities in
\eqref{eq:transformed-cone} are tight, i.e.,
\[
  D\relx{y}'=-\relx{\phi},
  \qquad\text{equivalently}\qquad
  \relx{y}'=-D^\top\relx{\phi}.
\]
Thus the simplicial cone~\eqref{eq:transformed-cone} can be rewritten as
$\relx{K}_{\JR} = \{\relx{y}\in\bR^2 : D(\relx{y}-\relx{y}') \ge 0\}$. Since
$D$ is orthogonal, the vectors $\relx{r}_1,\relx{r}_2$ are orthonormal; the
same vectors are therefore the extreme-ray directions, giving
\begin{equation*}
  \relx{K}_{\JR} = \bigl\{\, \relx{y}' + \relx{\eta}_1 \relx{r}_1 + \relx{\eta}_2 \relx{r}_2 : \relx{\eta}\in\bR^2_+ \,\bigr\},
\end{equation*}
with $i$-th boundary ray $\relx{R}_i \doteq \{\relx{y}' + \lambda \relx{r}_i : \lambda \ge 0\}$.

The sign of $q(\relx{y}')$ classifies the apex position relative to $C_{\Par}$: interior ($q(\relx{y}') < 0$), boundary ($q(\relx{y}') = 0$), or exterior ($q(\relx{y}') > 0$).
Along each boundary ray $\relx{R}_i$, the bowl function restricts to the quadratic (cf.~\eqref{eq:q-def})
\begin{equation}
  \label{eq:q-along-ray}
  h_i(\lambda) \;\doteq\; q(\relx{y}' + \lambda \relx{r}_i)
  \;=\; \nu_i \lambda^2 + \chi_i \lambda + q(\relx{y}'), \qquad \lambda \ge 0,
\end{equation}
where
\begin{equation}
  \label{eq:ray-coefficients}
  \nu_i \;\doteq\; -\relx{r}_{i2}^2/\Delta \ge 0, \qquad
  \chi_i  \;\doteq\; \relx{r}_{i1} - \frac{2\relx{y}'_2}{\Delta}\,\relx{r}_{i2},
\end{equation}
with discriminant $\kappa_i \doteq \chi_i^2 - 4\nu_i\,q(\relx{y}')$. The intersections of $\relx{R}_i$ with $\bd(C_{\Par})$ are exactly the nonneg\-ative roots of $h_i(\lambda) = 0$; real roots exist if and only if $\kappa_i \ge 0$.
A ray $\relx{R}_i$ is a \emph{recession ray} if its direction is a nonzero recession direction of $C_{\Par}$:  $\relx{r}_{i2}=0$ and $\relx{r}_{i1}<0$,  equivalently, $\nu_i=0$ and $\chi_i<0$. Since $\relx{r}_1,\relx{r}_2$ are orthonormal, at most one ray can be a recession ray.

\subsubsection{Joint-Range Hull Characterization}

We first consider the trivial case where $\relx{K}_{\JR} \subseteq \clos(C_{\Par}^c)$, i.e., $\relx{K}_{\JR}$ is entirely outside $\inter(C_{\Par})$, shown in Figure \ref{fig:assumption-cases}.

\begin{theorem}[Trivial case]
  \label{thm:trivial-case}
  Assume that the simplicial cone $\relx{K}_{\JR}\subseteq\clos(C_{\Par}^c)$.  If $\relx{S}_{\JR}=\bd(C_{\Par})$, then the joint-range set $\relx{Y}_{\JR}$ is empty or a singleton.  If $\relx{S}_{\JR}=\clos(C_{\Par}^c)$, then $\relx{Y}_{\JR}=\relx{K}_{\JR}$.
\end{theorem}
\begin{proof}
  The solid case is immediate from $\relx{K}_{\JR}\subseteq\clos(C_{\Par}^c)$.  In the boundary case, strict convexity of $C_{\Par}$ implies that a convex set contained in $\clos(C_{\Par}^c)$ can meet $\bd(C_{\Par})$ in at most one point.
\end{proof}

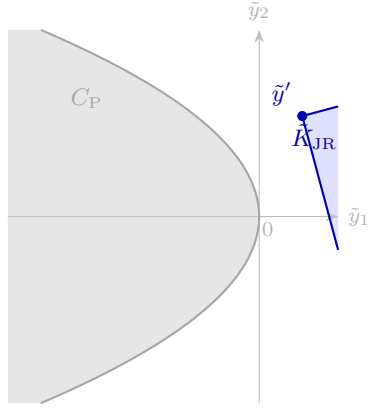
\begin{figure}[!htbp]
  \centering
  \resizebox{0.4\textwidth}{!}{
\begin{tikzpicture}[>=Stealth, scale=0.95, font=\small]
  \def\xmin{-3.5}\def\xmax{1.1}\def\ymin{-2.6}\def\ymax{2.6}\def\yclip{2.6}
  \coordinate (A) at (0.6,1.4);
  \coordinate (P1) at (1.1,-0.466025);
  \coordinate (P2) at (1.1,1.533975);
  \begin{scope}
    \clip (\xmin,\ymin) rectangle (\xmax,\ymax);
    \fill[gray!20] (\xmin,\ymin) -- (\xmin,\ymax)
      -- plot[variable=\t, domain=\yclip:-\yclip, samples=120] ({-0.45*\t*\t},{\t}) -- cycle;
    \fill[blue!12] (0.6,1.4) -- (1.1,-0.466025) -- (1.1,1.533975) -- cycle;
  \end{scope}
  \draw[gray!50, thin, ->] (\xmin,0) -- (\xmax,0) node[right, gray!60]{\scriptsize$\relx{y}_1$};
  \draw[gray!50, thin, ->] (0,\ymin) -- (0,\ymax) node[above, gray!60]{\scriptsize$\relx{y}_2$};
  \node[gray!60, font=\scriptsize] at (0.12,-0.18) {$0$};
  \draw[gray!70, thick] plot[variable=\t, domain=-\yclip:\yclip, samples=120] ({-0.45*\t*\t},{\t});
  \node[gray!70, font=\footnotesize] at (-2.4,1.65) {$C_{\Par}$};
  \draw[blue!70!black, thick] (A) -- (P1);
  \draw[blue!70!black, thick] (A) -- (P2);
  \fill[blue!70!black] (A) circle (2pt) node[above left, font=\footnotesize] {$\relx{y}'$};
  \node[blue!60!black, font=\footnotesize] at (0.766667,1.111325) {$\relx{K}_{\JR}$};
\end{tikzpicture}}
  \caption{Representative trivial configuration: $\relx{K}_{\JR}\subseteq \clos(C_{\Par}^c)$. The parabolic bowl $C_{\Par}$ is left of the parabola and the cone $\relx{K}_{\JR}$ is right of it.}
  \label{fig:assumption-cases}
\end{figure}

We next consider a special case  where $C_{\Par} \subseteq \relx{K}_{\JR}$, which is illustrated in Figure \ref{fig:containment-case}.

\begin{figure}[!htbp]
  \centering
  \resizebox{0.4\textwidth}{!}{
\begin{tikzpicture}[>=Stealth, scale=0.95, font=\small]
  \def\xmin{-3.5}\def\xmax{1.1}\def\ymin{-2.6}\def\ymax{2.6}\def\yclip{2.6}
  \coordinate (A) at (0.65,0);
  \coordinate (P1) at (-1.531659,2.6);
  \coordinate (P2) at (-1.531659,-2.6);
  \begin{scope}
    \clip (\xmin,\ymin) rectangle (\xmax,\ymax);
    \fill[gray!20] (\xmin,\ymin) -- (\xmin,\ymax)
      -- plot[variable=\t, domain=\yclip:-\yclip, samples=120] ({-0.45*\t*\t},{\t}) -- cycle;
    \fill[blue!12] (-3.5,-2.6) -- (-1.531659,-2.6) -- (0.65,0) -- (-1.531659,2.6) -- (-3.5,2.6) -- cycle;
  \end{scope}
  \draw[gray!50, thin, ->] (\xmin,0) -- (\xmax,0) node[right, gray!60]{\scriptsize$\relx{y}_1$};
  \draw[gray!50, thin, ->] (0,\ymin) -- (0,\ymax) node[above, gray!60]{\scriptsize$\relx{y}_2$};
  \node[gray!60, font=\scriptsize] at (0.12,-0.18) {$0$};
  \draw[gray!70, thick] plot[variable=\t, domain=-\yclip:\yclip, samples=120] ({-0.45*\t*\t},{\t});
  \node[gray!70, font=\footnotesize] at (-2.4,1.65) {$C_{\Par}$};
  \draw[blue!70!black, thick] (A) -- (P1);
  \draw[blue!70!black, thick] (A) -- (P2);
  \fill[blue!70!black] (A) circle (2pt) node[above right, font=\footnotesize] {$\relx{y}'$};
  \node[blue!60!black, font=\footnotesize] at (-0.616332,0) {$\relx{K}_{\JR}$};
\end{tikzpicture}}
  \caption{Representative configuration for the case $C_{\Par} \subseteq \relx{K}_{\JR}$.}
  \label{fig:containment-case}
\end{figure}
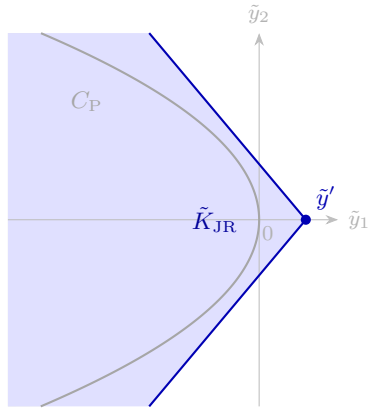

\begin{theorem}[Containment case]
  \label{thm:conv-hull-contained-bowl}
  Assume that \(C_{\Par}\subseteq \relx{K}_{\JR}\).
  \begin{itemize}
    \item If \(\relx{S}_{\JR}=\bd(C_{\Par})\), then $
            \clos(\conv(\relx{Y}_{\JR}))=C_{\Par}.
          $
    \item If \(\relx{S}_{\JR}=\clos(C_{\Par}^c)\), then
          $ \clos(\conv(\relx{Y}_{\JR}))=\relx{K}_{\JR}. $
  \end{itemize}
\end{theorem}
\begin{proof}
  Under the assumption $C_{\Par} \subseteq \relx{K}_{\JR}$, the boundary $\bd(C_{\Par})$ is contained in $\relx{K}_{\JR}$, thus the condition $\relx{K}_{\JR} \not\subseteq \clos(C_{\Par}^c)$ is satisfied. Since $\bd(C_{\Par})$ is also contained in $\relx{S}_{\JR}$ for both the parabola and solid region cases, we have $\bd(C_{\Par}) \subseteq \relx{Y}_{\JR}$. Consequently, the joint-range hull $\clos(\conv(\relx{Y}_{\JR})) \supseteq \conv(\bd(C_{\Par})) = C_{\Par}$.

  If $\relx{S}_{\JR} = \bd(C_{\Par})$, then the joint-range set $\relx{Y}_{\JR} = \bd(C_{\Par})$ and the joint-range hull $\clos(\conv(\relx{Y}_{\JR})) = C_{\Par}$. If $\relx{S}_{\JR}$ is the exterior region, then $\relx{K}_{\JR} = \relx{Y}_{\JR} \cup C_{\Par}$. Since $C_{\Par} \subseteq \clos(\conv(\relx{Y}_{\JR}))$ and $\relx{Y}_{\JR} \subseteq \clos(\conv(\relx{Y}_{\JR}))$, their union $\relx{K}_{\JR}$ is contained in $\clos(\conv(\relx{Y}_{\JR}))$, and the reverse inclusion is trivial.
\end{proof}

We next consider the cases where the apex is not in the interior of the parabolic bowl $C_{\Par}$ (illustrated by \Cref{fig:intersection-cases-outside}). We need the following useful lemma.

  \begin{lemma}
    \label{lem:recession-limit}
    Let $S\subseteq\bR^2$ contain a sequence $(p_k)$ with $\|p_k\|\to\infty$ and
    $p_k/\|p_k\|\to d$. Then $z+sd\in\clos(\conv(S))$ for every
    $z\in\clos(\conv(S))$ and every $s\ge0$.
  \end{lemma}
  \begin{proof}
    Let $\lambda_k\doteq s/\|p_k\|$, so that $\lambda_k\in[0,1]$ for all large
    $k$. The points $(1-\lambda_k)z+\lambda_k p_k$ lie in $\clos(\conv(S))$, and
    \[
      (1-\lambda_k)z+\lambda_k p_k
      = z+s\,\frac{p_k-z}{\|p_k\|}
      \;\longrightarrow\; z+sd,
    \]
    so the claim follows from the closedness of $\clos(\conv(S))$.
  \end{proof}

Then, we have the following result.

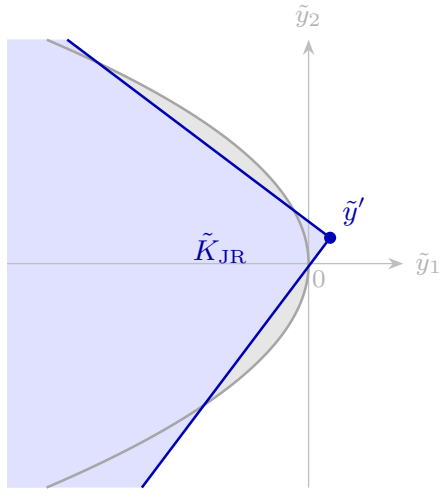
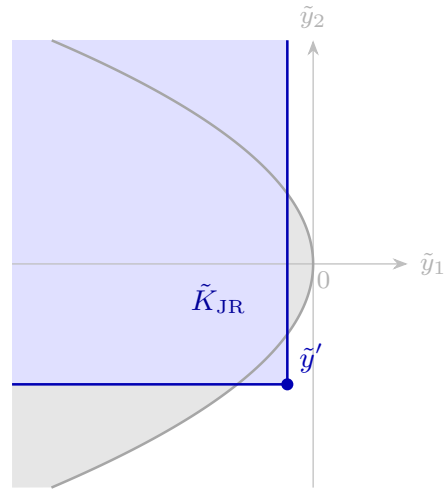
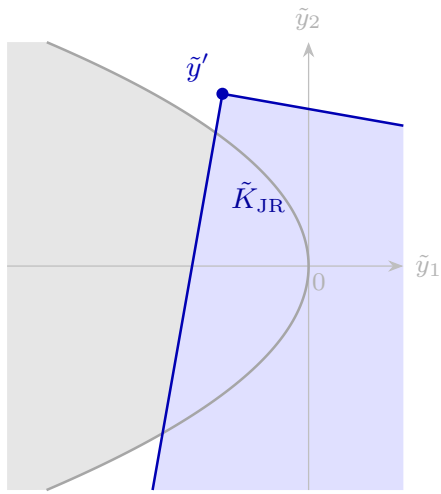
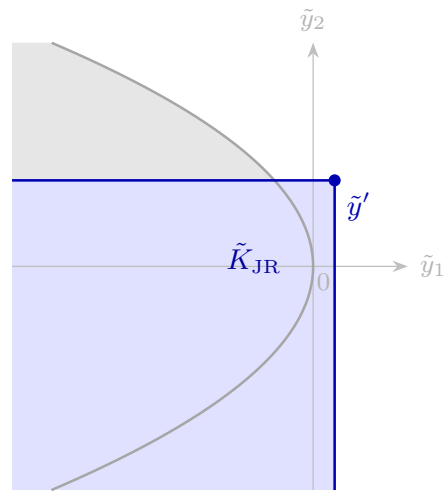
\begin{figure}[!htbp]
  \centering
  \begin{subfigure}[t]{0.48\textwidth}
    \centering
    \resizebox{\linewidth}{!}{
\begin{tikzpicture}[>=Stealth, scale=0.95, font=\small]
  \def\xmin{-3.5}\def\xmax{1.1}\def\ymin{-2.6}\def\ymax{2.6}\def\yclip{2.6}
  \coordinate (A) at (0.25,0.3);
  \coordinate (P1) at (-2.802203,2.6);
  \coordinate (P2) at (-1.935307,-2.6);
  \begin{scope}
    \clip (\xmin,\ymin) rectangle (\xmax,\ymax);
    \fill[gray!20] (\xmin,\ymin) -- (\xmin,\ymax)
      -- plot[variable=\t, domain=\yclip:-\yclip, samples=120] ({-0.45*\t*\t},{\t}) -- cycle;
    \fill[blue!12] (-3.5,-2.6) -- (-1.935307,-2.6) -- (0.25,0.3) -- (-2.802203,2.6) -- (-3.5,2.6) -- cycle;
  \end{scope}
  \draw[gray!50, thin, ->] (\xmin,0) -- (\xmax,0) node[right, gray!60]{\scriptsize$\relx{y}_1$};
  \draw[gray!50, thin, ->] (0,\ymin) -- (0,\ymax) node[above, gray!60]{\scriptsize$\relx{y}_2$};
  \node[gray!60, font=\scriptsize] at (0.12,-0.18) {$0$};
  \draw[gray!70, thick] plot[variable=\t, domain=-\yclip:\yclip, samples=120] ({-0.45*\t*\t},{\t});
  \draw[blue!70!black, thick] (A) -- (P1);
  \draw[blue!70!black, thick] (A) -- (P2);
  \fill[blue!70!black] (A) circle (2pt) node[above right, font=\footnotesize] {$\relx{y}'$};
  \node[blue!60!black, font=\footnotesize] at (-1.023751,0.18) {$\relx{K}_{\JR}$};
\end{tikzpicture}}
    \caption{No recession ray; both rays intersect.}
    \label{fig:intersection-cases-outside-both}
  \end{subfigure}\hfill
  \begin{subfigure}[t]{0.48\textwidth}
    \centering
    \resizebox{\linewidth}{!}{
\begin{tikzpicture}[>=Stealth, scale=0.95, font=\small]
  \def\xmin{-3.5}\def\xmax{1.1}\def\ymin{-2.6}\def\ymax{2.6}\def\yclip{2.6}
  \coordinate (A) at (-0.3,-1.4);
  \coordinate (P1) at (-0.3,2.6);
  \coordinate (P2) at (-3.5,-1.4);
  \begin{scope}
    \clip (\xmin,\ymin) rectangle (\xmax,\ymax);
    \fill[gray!20] (\xmin,\ymin) -- (\xmin,\ymax)
      -- plot[variable=\t, domain=\yclip:-\yclip, samples=120] ({-0.45*\t*\t},{\t}) -- cycle;
    \fill[blue!12] (-3.5,-1.4) -- (-0.3,-1.4) -- (-0.3,2.6) -- (-3.5,2.6) -- cycle;
  \end{scope}
  \draw[gray!50, thin, ->] (\xmin,0) -- (\xmax,0) node[right, gray!60]{\scriptsize$\relx{y}_1$};
  \draw[gray!50, thin, ->] (0,\ymin) -- (0,\ymax) node[above, gray!60]{\scriptsize$\relx{y}_2$};
  \node[gray!60, font=\scriptsize] at (0.12,-0.18) {$0$};
  \draw[gray!70, thick] plot[variable=\t, domain=-\yclip:\yclip, samples=120] ({-0.45*\t*\t},{\t});
  \draw[blue!70!black, thick] (A) -- (P1);
  \draw[blue!70!black, thick] (A) -- (P2);
  \fill[blue!70!black] (A) circle (2pt) node[above right, font=\footnotesize] {$\relx{y}'$};
  \node[blue!60!black, font=\footnotesize] at (-1.1,-0.4) {$\relx{K}_{\JR}$};
\end{tikzpicture}}
    \caption{Recession ray and one additional intersection.}
    \label{fig:intersection-cases-outside-recession-plus}
  \end{subfigure}

  \medskip
  \begin{subfigure}[t]{0.48\textwidth}
    \centering
    \resizebox{\linewidth}{!}{
\begin{tikzpicture}[>=Stealth, scale=0.95, font=\small]
  \def\xmin{-3.5}\def\xmax{1.1}\def\ymin{-2.6}\def\ymax{2.6}\def\yclip{2.6}
  \coordinate (A) at (-1,2);
  \coordinate (P1) at (-1.811104,-2.6);
  \coordinate (P2) at (1.1,1.629713);
  \begin{scope}
    \clip (\xmin,\ymin) rectangle (\xmax,\ymax);
    \fill[gray!20] (\xmin,\ymin) -- (\xmin,\ymax)
      -- plot[variable=\t, domain=\yclip:-\yclip, samples=120] ({-0.45*\t*\t},{\t}) -- cycle;
    \fill[blue!12] (-1,2) -- (-1.811104,-2.6) -- (1.1,-2.6) -- (1.1,1.629713) -- cycle;
  \end{scope}
  \draw[gray!50, thin, ->] (\xmin,0) -- (\xmax,0) node[right, gray!60]{\scriptsize$\relx{y}_1$};
  \draw[gray!50, thin, ->] (0,\ymin) -- (0,\ymax) node[above, gray!60]{\scriptsize$\relx{y}_2$};
  \node[gray!60, font=\scriptsize] at (0.12,-0.18) {$0$};
  \draw[gray!70, thick] plot[variable=\t, domain=-\yclip:\yclip, samples=120] ({-0.45*\t*\t},{\t});
  \draw[blue!70!black, thick] (A) -- (P1);
  \draw[blue!70!black, thick] (A) -- (P2);
  \fill[blue!70!black] (A) circle (2pt) node[above left, font=\footnotesize] {$\relx{y}'$};
  \node[blue!60!black, font=\footnotesize] at (-0.576388,0.803714) {$\relx{K}_{\JR}$};
\end{tikzpicture}}
    \caption{Two intersections on one ray, none on the other.}
    \label{fig:intersection-cases-outside-two-one-ray}
  \end{subfigure}\hfill
  \begin{subfigure}[t]{0.48\textwidth}
    \centering
    \resizebox{\linewidth}{!}{
\begin{tikzpicture}[>=Stealth, scale=0.95, font=\small]
  \def\xmin{-3.5}\def\xmax{1.1}\def\ymin{-2.6}\def\ymax{2.6}\def\yclip{2.6}
  \coordinate (A) at (0.25,1);
  \coordinate (P1) at (-3.5,1);
  \coordinate (P2) at (0.25,-2.6);
  \begin{scope}
    \clip (\xmin,\ymin) rectangle (\xmax,\ymax);
    \fill[gray!20] (\xmin,\ymin) -- (\xmin,\ymax)
      -- plot[variable=\t, domain=\yclip:-\yclip, samples=120] ({-0.45*\t*\t},{\t}) -- cycle;
    \fill[blue!12] (0.25,1) -- (-3.5,1) -- (-3.5,-2.6) -- (0.25,-2.6) -- cycle;
  \end{scope}
  \draw[gray!50, thin, ->] (\xmin,0) -- (\xmax,0) node[right, gray!60]{\scriptsize$\relx{y}_1$};
  \draw[gray!50, thin, ->] (0,\ymin) -- (0,\ymax) node[above, gray!60]{\scriptsize$\relx{y}_2$};
  \node[gray!60, font=\scriptsize] at (0.12,-0.18) {$0$};
  \draw[gray!70, thick] plot[variable=\t, domain=-\yclip:\yclip, samples=120] ({-0.45*\t*\t},{\t});
  \draw[blue!70!black, thick] (A) -- (P1);
  \draw[blue!70!black, thick] (A) -- (P2);
  \fill[blue!70!black] (A) circle (2pt) node[below right, font=\footnotesize] {$\relx{y}'$};
  \node[blue!60!black, font=\footnotesize] at (-0.6875,0.1) {$\relx{K}_{\JR}$};
\end{tikzpicture}}
    \caption{One recession-ray intersection only.}
    \label{fig:intersection-cases-outside-recession-only}
  \end{subfigure}
  \caption{Representative configurations for the outside-apex cases ($q(\relx{y}') > 0$).}
  \label{fig:intersection-cases-outside}
\end{figure}

\begin{theorem}[Non-interior-apex case]
  \label{thm:closure-conv-diff}
  Assume that $\relx{K}_{\JR} \not\subseteq \clos(C_{\Par}^c)$, \(C_{\Par}\not\subseteq \relx{K}_{\JR}\), and $q(\relx{y}')\ge 0$. Then
  \begin{itemize}
    \item If $\relx{S}_{\JR}=\bd(C_{\Par})$, then $\clos(\conv(\relx{Y}_{\JR}))= \relx{K}_{\JR}\cap C_{\Par}$;
    \item  If $\relx{S}_{\JR}=\clos(C_{\Par}^c)$, then $\clos(\conv(\relx{Y}_{\JR}))= \relx{K}_{\JR}$.
  \end{itemize}
\end{theorem}
\begin{proof}
  We first prove that $\clos(\conv(\bd(C_{\Par}) \cap \relx{K}_{\JR})) = H_{\rm loc} \doteq \relx{K}_{\JR}\cap C_{\Par}$. Since $H_{\rm loc}$ is closed and convex and contains $\bd(C_{\Par}) \cap \relx{K}_{\JR}$, we have $\clos(\conv(\bd(C_{\Par}) \cap \relx{K}_{\JR})) \subseteq H_{\rm loc}$. Note that the condition $\relx{K}_{\JR} \not\subseteq \clos(C_{\Par}^c)$ makes $H_{\rm loc}$ non-empty. We next show the reverse inclusion.

  Note that the boundary $\bd(H_{\rm loc})$ is a subset of $\bd(\relx{K}_{\JR}) \cup \bd(C_{\Par})$. Also, $\bd(\relx{K}_{\JR})$ consists of the two boundary rays $\relx{R}_1$ and $\relx{R}_2$ starting from the apex $\relx{y}'$.

  If a ray $\relx{R}_i$ is not a recession ray, since $q(\relx{y}')\ge 0$ (the apex is not in the interior of $C_{\Par}$), the ray can have at most one continuous line segment intersecting $C_{\Par}$. See for example \Cref{fig:intersection-cases-outside}\subref{fig:intersection-cases-outside-both}--\subref{fig:intersection-cases-outside-two-one-ray}.
  If this line segment exists, it is the convex combination of its endpoints, which are intersection points of $\relx{R}_i$ with $\bd(C_{\Par})$ (thus in $\bd(C_{\Par}) \cap \relx{K}_{\JR}$). Consequently, this segment lies in $\conv(\bd(C_{\Par}) \cap \relx{K}_{\JR})$.

  Alternatively, if a recession ray (say $\relx{R}_1$) exists, then $\relx{r}_{12}=0$ and $\relx{r}_{11}<0$, so $h_1(\lambda)=q(\relx{y}')+\lambda \relx{r}_{11}$ has a unique nonnegative root. Let $\relx{v}$ be the corresponding intersection with $\bd(C_{\Par})$, so \(\relx{v}_1=\relx{v}_2^2/\Delta\). See for example \Cref{fig:intersection-cases-outside}\subref{fig:intersection-cases-outside-recession-plus}--\subref{fig:intersection-cases-outside-recession-only}. Then, the intersection of the horizontal $\relx{R}_1$ with $C_{\Par}$ is a ray $\ell =\{\relx{v}-se_1:s\ge0\} \subseteq \relx{R}_1$ starting at $\relx{v} \in \bd(C_{\Par})\cap\relx{K}_{\JR}$.  Since \(\relx{r}_1=(-1,0)\) and \(\relx{r}_2\perp\relx{r}_1\) is a unit vector, we have \(\relx{r}_2=(0,\sigma)\) for a sign \(\sigma\in\{\pm1\}\), and
  membership of any point $\relx{y}$  in \(\relx{K}_{\JR}\) amounts to \(\relx{r}_i^\top(\relx{y}-\relx{y}')\ge0\) for \(i=1,2\). For \(t\ge0\), define
  \[
  p(t)\doteq \left(\frac{(\relx{v}_2+\sigma t)^2}{\Delta},\ \relx{v}_2+\sigma t\right)
  \in \bd(C_{\Par}).
  \]
  Because \(\relx{R}_1\) is horizontal, \(\relx{v}_2=\relx{y}'_2\), so \(\relx{r}_2^\top(p(t)-\relx{y}')=\sigma\bigl(p_2(t)-\relx{y}'_2\bigr)=t\ge0\) for
  all \(t\ge0\). Moreover, \(q(\relx{y}')\ge0\) gives \(\relx{y}'_1\ge \relx{v}_2^2/\Delta=\relx{v}_1\), hence
  \[
  \relx{r}_1^\top(p(t)-\relx{y}')
  =\relx{y}'_1-p_1(t)
  \ge \relx{v}_1-p_1(t)
  =-\frac{2\sigma \relx{v}_2 t+t^2}{\Delta}
  \ \ge\ 0
  \qquad\text{for all } t\ge 2|\relx{v}_2|,
  \]
  since \(\Delta<0\). Therefore
  \[
  p(t)\in \bd(C_{\Par})\cap\relx{K}_{\JR}
  \qquad\text{for all } t\ge 2|\relx{v}_2|.
  \]
Since \(p_1(t)\sim t^2/\Delta\to-\infty\) whereas \(p_2(t)\) grows only
  linearly in \(t\), we have \(\|p(t)\|\to\infty\) and
  \(p(t)/\|p(t)\|\to-e_1\). Applying \Cref{lem:recession-limit} to
  \(S=\bd(C_{\Par})\cap\relx{K}_{\JR}\), with the sequence \(p(t_k)\),
  \(t_k\to\infty\), \(z=\relx{v}\in S\), and \(d=-e_1\), yields
  \(\relx{v}-se_1\in\clos(\conv(\bd(C_{\Par})\cap\relx{K}_{\JR}))\) for every
  \(s\ge0\), i.e.,
  \(\ell\subseteq\clos(\conv(\bd(C_{\Par})\cap\relx{K}_{\JR}))\).

  Thus, the boundary $\bd(H_{\rm loc})$ consists entirely of parabolic arcs from $\bd(C_{\Par})\cap \relx{K}_{\JR}$ and line segments or rays contained in $\clos(\conv(\bd(C_{\Par}) \cap \relx{K}_{\JR}))$. Therefore, $\bd(H_{\rm loc}) \subseteq \clos(\conv(\bd(C_{\Par}) \cap \relx{K}_{\JR}))$. We now obtain the remaining points of $H_{\rm loc}$ by a direct planar argument. Let $z\in H_{\rm loc}$. If $z\in\bd(H_{\rm loc})$, the previous inclusion applies. Otherwise, choose a direction $w$ such that neither $w$ nor $-w$ is a recession direction of $H_{\rm loc}$; such a direction exists because $\operatorname{rec}(H_{\rm loc})\subseteq\operatorname{rec}(C_{\Par})$, and the parabolic bowl has at most a one-dimensional recession cone. The section $I\doteq\{t\in\bR:z+tw\in H_{\rm loc}\}$ is a closed interval containing $0$. Since neither $w$ nor $-w$ is a recession direction, this interval is bounded, say $I=[t_-,t_+]$ with $t_-<0<t_+$. Its endpoints give $z^-=z+t_-w$ and $z^+=z+t_+w$ in $\bd(H_{\rm loc})$, and $z\in[z^-,z^+]$. Since both endpoints lie in $\clos(\conv(\bd(C_{\Par}) \cap \relx{K}_{\JR}))$, and this set is closed and convex, $z$ lies in it as well. Hence $H_{\rm loc} \subseteq \clos(\conv(\bd(C_{\Par}) \cap \relx{K}_{\JR}))$.

  We then prove the second identity $\clos(\conv(\relx{K}_{\JR}\setminus \inter(C_{\Par})))=\relx{K}_{\JR}$. Since $\relx{K}_{\JR}$ is closed and convex, $\clos(\conv(\relx{K}_{\JR}\setminus \inter(C_{\Par}))) \subseteq\relx{K}_{\JR}$. We next show the reverse inclusion.

  If a ray $\relx{R}_i$ is not a recession ray, then either $\relx{R}_i\cap\inter(C_{\Par})=\varnothing$, in which case $\relx{R}_i\subseteq \clos(\relx{K}_{\JR}\setminus \inter(C_{\Par}))$, or $\relx{R}_i\cap\inter(C_{\Par})$ is a bounded open segment. In the latter case, the endpoints of this segment lie in $\relx{R}_i\setminus \inter(C_{\Par})$, so the segment is contained in the convex hull of $\relx{R}_i\setminus \inter(C_{\Par})$. Hence $\relx{R}_i \subseteq \clos(\conv(\relx{K}_{\JR}\setminus \inter(C_{\Par})))$.

  Alternatively, if a recession ray (say $\relx{R}_1$) exists, let $\relx{v}$ be its unique intersection with $\bd(C_{\Par})$. Then, the intersection of $\relx{R}_1$ with $\inter(C_{\Par})$ is an unbounded open ray. However, as shown earlier, the closed ray $\ell \subseteq \relx{R}_1 \cap C_{\Par}$ satisfies $\ell \subseteq \clos(\conv(\bd(C_{\Par}) \cap \relx{K}_{\JR}))$. Since $\bd(C_{\Par}) \cap \relx{K}_{\JR} \subseteq \relx{K}_{\JR}\setminus\inter(C_{\Par})$, we have $\ell \subseteq \clos(\conv(\relx{K}_{\JR}\setminus \inter(C_{\Par})))$. Moreover, the remaining segment $\relx{R}_1 \smallsetminus \ell$ lies in $\relx{R}_1 \setminus \inter(C_{\Par}) \subseteq \clos(\conv(\relx{K}_{\JR}\setminus \inter(C_{\Par})))$. This shows that $\relx{R}_i \subseteq \clos(\conv(\relx{K}_{\JR}\setminus \inter(C_{\Par})))$ for all $i=1,2$.

  Since the boundary rays $\relx{R}_1, \relx{R}_2$ and the arc $\bd(C_{\Par}) \cap \relx{K}_{\JR}$ are all contained in $\clos(\conv(\relx{K}_{\JR}\setminus \inter(C_{\Par})))$, their closed convex hull, which is the entire cone $\relx{K}_{\JR}$, must also be contained within it. Hence $\relx{K}_{\JR} \subseteq \clos(\conv(\relx{K}_{\JR}\setminus \inter(C_{\Par})))$.

\end{proof}

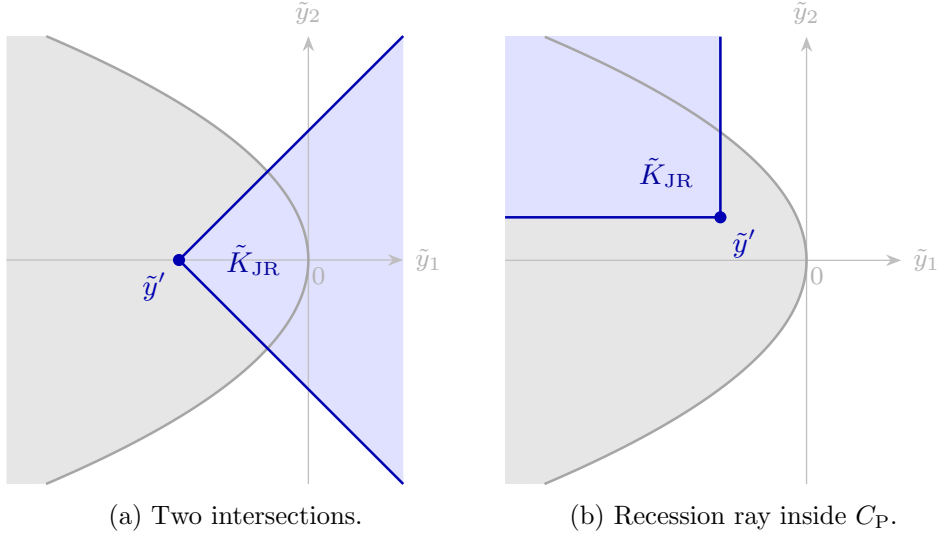
\begin{figure}[!htbp]
  \centering
  \begin{subfigure}[t]{0.48\textwidth}
    \centering
    \resizebox{\linewidth}{!}{
\begin{tikzpicture}[>=Stealth, scale=0.95, font=\small]
  \def\xmin{-3.5}\def\xmax{1.1}\def\ymin{-2.6}\def\ymax{2.6}\def\yclip{2.6}
  \coordinate (A) at (-1.5,0);
  \coordinate (P1) at (1.1,-2.6);
  \coordinate (P2) at (1.1,2.6);
  \begin{scope}
    \clip (\xmin,\ymin) rectangle (\xmax,\ymax);
    \fill[gray!20] (\xmin,\ymin) -- (\xmin,\ymax)
      -- plot[variable=\t, domain=\yclip:-\yclip, samples=120] ({-0.45*\t*\t},{\t}) -- cycle;
    \fill[blue!12] (-1.5,0) -- (1.1,-2.6) -- (1.1,2.6) -- cycle;
  \end{scope}
  \draw[gray!50, thin, ->] (\xmin,0) -- (\xmax,0) node[right, gray!60]{\scriptsize$\relx{y}_1$};
  \draw[gray!50, thin, ->] (0,\ymin) -- (0,\ymax) node[above, gray!60]{\scriptsize$\relx{y}_2$};
  \node[gray!60, font=\scriptsize] at (0.12,-0.18) {$0$};
  \draw[gray!70, thick] plot[variable=\t, domain=-\yclip:\yclip, samples=120] ({-0.45*\t*\t},{\t});
  \draw[blue!70!black, thick] (A) -- (P1);
  \draw[blue!70!black, thick] (A) -- (P2);
  \fill[blue!70!black] (A) circle (2pt) node[below left, font=\footnotesize] {$\relx{y}'$};
  \node[blue!60!black, font=\footnotesize] at (-0.633333,0) {$\relx{K}_{\JR}$};
\end{tikzpicture}}
    \caption{Two intersections.}
    \label{fig:intersection-cases-inside-one-one}
  \end{subfigure}\hfill
  \begin{subfigure}[t]{0.48\textwidth}
    \centering
    \resizebox{\linewidth}{!}{
\begin{tikzpicture}[>=Stealth, scale=0.95, font=\small]
  \def\xmin{-3.5}\def\xmax{1.1}\def\ymin{-2.6}\def\ymax{2.6}\def\yclip{2.6}
  \coordinate (A) at (-1,0.5);
  \coordinate (P1) at (-1,2.6);
  \coordinate (P2) at (-3.5,0.5);
  \begin{scope}
    \clip (\xmin,\ymin) rectangle (\xmax,\ymax);
    \fill[gray!20] (\xmin,\ymin) -- (\xmin,\ymax)
      -- plot[variable=\t, domain=\yclip:-\yclip, samples=120] ({-0.45*\t*\t},{\t}) -- cycle;
    \fill[blue!12] (-3.5,0.5) -- (-1,0.5) -- (-1,2.6) -- (-3.5,2.6) -- cycle;
  \end{scope}
  \draw[gray!50, thin, ->] (\xmin,0) -- (\xmax,0) node[right, gray!60]{\scriptsize$\relx{y}_1$};
  \draw[gray!50, thin, ->] (0,\ymin) -- (0,\ymax) node[above, gray!60]{\scriptsize$\relx{y}_2$};
  \node[gray!60, font=\scriptsize] at (0.12,-0.18) {$0$};
  \draw[gray!70, thick] plot[variable=\t, domain=-\yclip:\yclip, samples=120] ({-0.45*\t*\t},{\t});
  \draw[blue!70!black, thick] (A) -- (P1);
  \draw[blue!70!black, thick] (A) -- (P2);
  \fill[blue!70!black] (A) circle (2pt) node[below right, font=\footnotesize] {$\relx{y}'$};
  \node[blue!60!black, font=\footnotesize] at (-1.625,1.025) {$\relx{K}_{\JR}$};
\end{tikzpicture}}
    \caption{Recession ray inside $C_{\Par}$.}
    \label{fig:intersection-cases-inside-recession}
  \end{subfigure}
  \caption{Representative configurations for the interior-apex case ($q(\relx{y}') < 0$).}
  \label{fig:intersection-cases-inside}
\end{figure}

If the apex is in the interior of the parabolic bowl $C_{\Par}$ (illustrated by \Cref{fig:intersection-cases-inside}), we have the following result.

\begin{theorem}[Interior-apex case]
  \label{thm:conv-hull-interior-apex}
  Assume that $\relx{K}_{\JR} \not\subseteq \clos(C_{\Par}^c)$, \(C_{\Par}\not\subseteq \relx{K}_{\JR}\) (\(q(\relx{y}')<0\)). Exactly one of the following holds:
  \begin{description}
    \item[\textbf{(i)}]
          Each boundary ray \(\relx{R}_i\) meets \(\bd(C_{\Par})\) at a unique point \(\relx{v}_i\) (see for example \Cref{fig:intersection-cases-inside}\subref{fig:intersection-cases-inside-one-one}). Let \(H_{12}\) be the closed half-space bounded by the line through \(\relx{v}_1\) and \(\relx{v}_2\) that does not contain the apex \(\relx{y}'\). Then:
          \begin{itemize}
            \item if \(\relx{S}_{\JR} = \bd(C_{\Par})\), then $\clos(\conv(\relx{Y}_{\JR})) = C_{\Par} \cap H_{12}$.
            \item if \(\relx{S}_{\JR} = \clos(C_{\Par}^c)\), then $\clos(\conv(\relx{Y}_{\JR})) = \relx{K}_{\JR} \cap H_{12}$.
          \end{itemize}

    \item[\textbf{(ii)}]
          One boundary ray \(\relx{R}_{i'}\) is a recession ray contained in \(C_{\Par}\), and the other boundary ray \(\relx{R}_i\) meets \(\bd(C_{\Par})\) at a unique point \(\relx{v}\) (see for example \Cref{fig:intersection-cases-inside}\subref{fig:intersection-cases-inside-recession}). Let \(H_{\relx{v}}\) be the closed half-space bounded by the line through \(\relx{v}\) parallel to \(\relx{R}_{i'}\) that does not contain the apex \(\relx{y}'\). Then:
          \begin{itemize}
            \item if \(\relx{S}_{\JR} = \bd(C_{\Par})\), then $\clos(\conv(\relx{Y}_{\JR})) = C_{\Par} \cap H_{\relx{v}}$.
            \item if \(\relx{S}_{\JR} = \clos(C_{\Par}^c)\), then $\clos(\conv(\relx{Y}_{\JR})) = \relx{K}_{\JR} \cap H_{\relx{v}}$.
          \end{itemize}
  \end{description}
\end{theorem}
\begin{proof}
  Regarding \eqref{eq:ray-coefficients},
  at most one \(\nu_i\) can vanish since \(\relx{r}_1,\relx{r}_2\) are orthonormal. If \(\nu_i=0\), then \(\chi_i\neq 0\) because \(\relx{r}_i\neq 0\). For any \(i\), the condition \(q(\relx{y}')<0\) implies
  \[
    \kappa_i=\chi_i^2-4\nu_i q(\relx{y}')\ge 0,
  \]
  so \(h_i(\lambda)=0\) always has real roots. Since \(h_i(0)=q(\relx{y}')<0\) and $h_i$ is convex as a function of $\lambda$ (because $\nu_i \ge 0$), at least one root must be positive. If \(\nu_i>0\), then \(h_i(\lambda)\to +\infty\) as \(\lambda\to +\infty\), so \(\relx{R}_i\) has exactly one intersection point with \(\bd(C_{\Par})\). If \(\nu_i=0\), then
  \[
    h_i(\lambda)=\chi_i\lambda+q(\relx{y}'),
  \]
  and this has a positive root if and only if \(\chi_i>0\); if \(\chi_i<0\), then \(h_i(\lambda)<0\) for all \(\lambda\ge 0\), and therefore \(\relx{R}_i\subseteq C_{\Par}\), which shows that \(\relx{R}_i\) is a recession ray. Hence, up to relabeling, either each ray meets \(\bd(C_{\Par})\) exactly once, giving case \textbf{(i)}, or one ray is a recession ray contained in \(C_{\Par}\) whereas the other meets \(\bd(C_{\Par})\) exactly once, giving case \textbf{(ii)}.

  \textbf{Case (i).}
  If \(\relx{S}_{\JR} = \bd(C_{\Par})\), then the set \(\relx{Y}_{\JR}\) is exactly the bounded parabolic arc of \(\bd(C_{\Par})\) joining \(\relx{v}_1\) and \(\relx{v}_2\). Since \(q(\relx{y}')<0\), the apex lies on the opposite side of the chord from the arc. Hence the half-space \(H_{12}\) excluding the apex is exactly the half-space containing the arc, and the convex hull of this arc is the convex region enclosed by the arc and its chord \([\relx{v}_1,\relx{v}_2]\), namely \(C_{\Par} \cap H_{12}\).

 If \(\relx{S}_{\JR} = \clos(C_{\Par}^c)\), then
\(\relx{Y}_{\JR}=\relx{K}_{\JR}\cap\clos(C_{\Par}^c)\).   We prove the identity
  \(\clos(\conv(\relx{Y}_{\JR}))=\relx{K}_{\JR}\cap H_{12}\) by two inclusions.
  We first check that \(\relx{Y}_{\JR}\subseteq H_{12}\). The part of the cone
  outside \(H_{12}\) satisfies
  \(\relx{K}_{\JR}\setminus H_{12}\subseteq\conv\{\relx{y}',\relx{v}_1,\relx{v}_2\}\subseteq C_{\Par}\),
  where the second inclusion follows from the convexity of \(C_{\Par}\) and
  \(\relx{y}',\relx{v}_1,\relx{v}_2\in C_{\Par}\). Hence any
  \(\relx{z}\in\relx{Y}_{\JR}\setminus H_{12}\) would satisfy both
  \(q(\relx{z})\ge0\) and \(q(\relx{z})\le0\), so
  \(\relx{z}\in\bd(C_{\Par})\cap\relx{K}_{\JR}\); but this set is the parabolic
  arc joining \(\relx{v}_1\) and \(\relx{v}_2\), which is contained in
  \(H_{12}\), a contradiction. Since \(\relx{K}_{\JR}\cap H_{12}\) is closed and convex and contains
  \(\relx{Y}_{\JR}\),  we immediately have
\[
\clos(\conv(\relx{Y}_{\JR}))\subseteq \relx{K}_{\JR}\cap H_{12}.
\]
For the reverse inclusion, the boundary-case argument above gives
    \(C_{\Par}\cap H_{12}\subseteq\clos(\conv(\bd(C_{\Par})\cap \relx{K}_{\JR}))
    \subseteq\clos(\conv(\relx{Y}_{\JR}))\), because
    \(\bd(C_{\Par})\subseteq\clos(C_{\Par}^c)\). Moreover, any
    \(\relx{z}\in\relx{K}_{\JR}\cap H_{12}\) satisfies either
    \(q(\relx{z})\le0\), so that \(\relx{z}\in C_{\Par}\cap H_{12}\), or
    \(q(\relx{z})\ge0\), so that
    \(\relx{z}\in\relx{K}_{\JR}\cap\clos(C_{\Par}^c)=\relx{Y}_{\JR}\).
    Both sets lie in \(\clos(\conv(\relx{Y}_{\JR}))\), hence
    \(\relx{K}_{\JR}\cap H_{12}\subseteq\clos(\conv(\relx{Y}_{\JR}))\),
    and the desired equality follows.

  \textbf{Case (ii).}
  Let \(\relx{R}_{i'}\) be the recession ray. Since \(\relx{r}_{i'2}=0\) and \(\relx{r}_i \perp \relx{r}_{i'}\), the other ray \(\relx{R}_i\) is vertical, whereas \(\relx{R}_{i'}\) is horizontal. Hence \(\bd(C_{\Par})\cap \relx{K}_{\JR}\) is a one-sided branch of the parabola issuing from \(\relx{v}\). Because \(q(\relx{y}')<0\), the apex lies on the side of the line through \(\relx{v}\) parallel to \(\relx{R}_{i'}\) opposite to that branch. Therefore the half-space \(H_{\relx{v}}\) excluding the apex is exactly the half-space containing \(\bd(C_{\Par})\cap \relx{K}_{\JR}\).

  If \(\relx{S}_{\JR} = \bd(C_{\Par})\), then the closed convex hull of this one-sided branch is precisely \(C_{\Par} \cap H_{\relx{v}}\). To see this,
keep the canonical coordinates \(\relx{y}=(\relx{y}_1,\relx{y}_2)\). Writing
  \(\sigma\in\{\pm1\}\) for the direction of the vertical ray, we have
  \(\relx{K}_{\JR}=\{\relx{y}:\relx{y}_1\le\relx{y}'_1,\ \sigma(\relx{y}_2-\relx{y}'_2)\ge0\}\).
  Since \(q(\relx{y}')<0\), the vertical line \(\relx{y}_1=\relx{y}'_1\) meets
  \(\bd(C_{\Par})\) at the two heights \(\pm\sqrt{\Delta\relx{y}'_1}\), with
  \(\relx{y}'_2\) strictly between them; hence
  \(\relx{v}=(\relx{y}'_1,\ \sigma\sqrt{\Delta\relx{y}'_1})\), i.e.,
  \(\sigma\relx{v}_2=|\relx{v}_2|\). Therefore
  \[
  p(t)\doteq \left(\frac{(\relx{v}_2+\sigma t)^2}{\Delta},\ \relx{v}_2+\sigma t\right),
  \qquad t\ge0,
  \]
  parameterizes the one-sided branch \(\bd(C_{\Par})\cap\relx{K}_{\JR}\):
  indeed, \(\sigma p_2(t)=|\relx{v}_2|+t\), so \(p_2(t)^2\ge\relx{v}_2^2\) gives
  \(p_1(t)\le\relx{y}'_1\) (as \(\Delta<0\)), and
  \(\sigma(p_2(t)-\relx{y}'_2)>0\).
Then \(H_{\relx{v}}=\{\sigma(\relx{y}_2-\relx{v}_2)\ge0\}\). The inclusion
\(\clos(\conv(\bd(C_{\Par})\cap\relx{K}_{\JR}))\subseteq C_{\Par}\cap H_{\relx{v}}\)
follows because the latter set is closed and convex and contains the branch.
Conversely, every point of \(C_{\Par}\cap H_{\relx{v}}\) can be written in the
form \(p(\beta)-se_1\) for some \(\beta\ge0\) and \(s\ge0\): indeed,
\(H_{\relx{v}}=\{\sigma(\relx{y}_2-\relx{v}_2)\ge0\}\) gives
\(\relx{y}_2=\relx{v}_2+\sigma\beta\) with \(\beta\ge0\), and then
\(q(\relx{y})\le0\) is equivalent to \(\relx{y}_1\le \relx{y}_2^2/\Delta\),
i.e., \(\relx{y}_1=(\relx{v}_2+\sigma\beta)^2/\Delta-s\) for some \(s\ge0\).
Since \(p_1(t)\sim t^2/\Delta\to-\infty\) whereas \(p_2(t)\) grows only
  linearly, \(\|p(t)\|\to\infty\) and \(p(t)/\|p(t)\|\to-e_1\). Applying
  \Cref{lem:recession-limit} to \(S=\bd(C_{\Par})\cap\relx{K}_{\JR}\), with
  \(z=p(\beta)\in S\) and \(d=-e_1\), yields
  \(p(\beta)-se_1\in\clos(\conv(\bd(C_{\Par})\cap\relx{K}_{\JR}))\) for all
  \(\beta\ge0\) and \(s\ge0\). Hence
  \(C_{\Par}\cap H_{\relx{v}}\subseteq\clos(\conv(\bd(C_{\Par})\cap\relx{K}_{\JR}))\),
  proving the boundary case.

If \(\relx{S}_{\JR}=\clos(C_{\Par}^c)\), then
\(\relx{Y}_{\JR}=\relx{K}_{\JR}\cap\clos(C_{\Par}^c)\).  Note that
  \(\relx{Y}_{\JR}\subseteq H_{\relx{v}}\), since
  \(\relx{K}_{\JR}\setminus H_{\relx{v}}
  \subseteq\{\relx{y}:\relx{y}_1\le\relx{y}'_1,\ |\relx{y}_2|<|\relx{v}_2|\}
  \subseteq\inter(C_{\Par})\)
  (for such \(\relx{y}\),
  \(q(\relx{y})=\relx{y}_1-\relx{y}_2^2/\Delta<\relx{y}'_1-\relx{v}_2^2/\Delta=0\)). The set
\(\relx{K}_{\JR}\cap H_{\relx{v}}\) is closed and convex and contains
\(\relx{Y}_{\JR}\), so
\(\clos(\conv(\relx{Y}_{\JR}))\subseteq\relx{K}_{\JR}\cap H_{\relx{v}}\).
For the reverse inclusion, the boundary case gives
\(C_{\Par}\cap H_{\relx{v}}\subseteq\clos(\conv(\relx{Y}_{\JR}))\), and the
remaining points of \(\relx{K}_{\JR}\cap H_{\relx{v}}\) lie in
\(\relx{K}_{\JR}\cap\clos(C_{\Par}^c)=\relx{Y}_{\JR}\). Therefore
\(\relx{K}_{\JR}\cap H_{\relx{v}}\subseteq\clos(\conv(\relx{Y}_{\JR}))\), and
equality follows.

\end{proof}

\subsection{Separation of Joint-Range Inequalities in the Nonconvex Case}
\label{sec:separation-parabolic-nonconvex}

Since the joint-range inequalities are, by definition, the valid inequalities of $H_{\JR}$, separation reduces to finding a most violated supporting inequality of $H_{\JR}$. In the nontrivial cases above, the convex hull in canonical coordinates is obtained as the intersection of one or several elementary convex sets: a polyhedron or the parabolic bowl \(C_{\Par}\). Hence, given a query point \(\relx{y}^*\) to cut off, it suffices to compute a supporting inequality for each elementary convex set and then choose the best one according to some criterion, for example the Euclidean distance from \(\relx{y}^*\) to the corresponding half-space.

For the polyhedral part, separation is immediate: one checks the supporting inequalities of the corresponding polyhedron and selects the best candidate among them.

\begin{remark}[Secant joint-range inequalities and intersection cuts]
  \label{rem:parabolic-secant-intersection-cuts}
  In the interior-apex case of \Cref{thm:conv-hull-interior-apex}, the parabolic bowl $C_{\Par}$ is an $\relx{S}_{\JR}$-free set whose interior contains the apex of $\relx{K}_{\JR}$.  The half-spaces $H_{12}$ and $H_{\relx{v}}$ are defined by the points where the boundary rays of $\relx{K}_{\JR}$ leave $C_{\Par}$.  These are exactly the intersection-cut step-length points, so the valid inequalities are intersection cuts generated by $(\relx{K}_{\JR},C_{\Par})$. In particular, we call them secant joint-range inequalities.
\end{remark}

For the parabolic bowl \(C_{\Par}\), it suffices to separate the tangent inequalities of \(\bd(C_{\Par})\). Every point on
\[
  \bd(C_{\Par})=\{\relx{y}_1=\relx{y}_2^2/\Delta\}
\]
can be parameterized as \((\tau^2/\Delta,\tau)\), \(\tau\in\bR\). The closest point on \(\bd(C_{\Par})\) to \(\relx{y}^*\) is obtained by minimizing
\[
  f(\tau)=\Bigl(\tfrac{\tau^2}{\Delta}-\relx{y}_1^*\Bigr)^2+(\tau-\relx{y}_2^*)^2.
\]
Setting \(f'(\tau)=0\) and multiplying by \(\Delta^2\) yields the cubic equation
\begin{equation}
  \label{eq:cubic-proj}
  2\tau^3+(\Delta^2-2\Delta \relx{y}_1^*)\tau-\Delta^2 \relx{y}_2^*=0.
\end{equation}
Among the real roots of \eqref{eq:cubic-proj}, we choose the one minimizing \(f(\tau)\). If \(\relx{v}=(\tau^2/\Delta,\tau)\) is the corresponding closest point, then the supporting tangent inequality for \(C_{\Par}\) at \(\relx{v}\) is
\[
  \relx{y}_1-\frac{2\tau}{\Delta}\relx{y}_2 \le -\frac{\tau^2}{\Delta}.
\]

After generating all candidate projected inequalities, we select the best supporting inequality $  \mu^\top \relx{y} \le \mu_0$ in the two-dimensional image space and then lift it back to the extended QCQP formulation by substituting the affine map $\cA_{\JR}$ (and, in canonical coordinates, the inverse transform $T^{-1}$):
\begin{equation}
  \mu^\top T^{-1}(\cA_{\JR}(X,x)) = \mu^\top(D^{-1}\cA_{\JR}(X,x) -c) \le \mu_0.
\end{equation}
This completes the project-then-lift pipeline: the geometry is handled in two dimensions, whereas the final valid inequality is expressed in the extended formulation and inherits the sparsity of the base inequalities.

\subsection{Separation of Joint-Range Inequalities in the Convex Case}

When $S_{\JR}$ is convex, $H_{\JR} = \clos(S_{\JR} \cap K_{\JR})$, so every supporting half-space of $S_{\JR}$ yields a joint-range inequality. Beyond the two base inequalities defining $K_{\JR}$, supporting inequalities of $S_{\JR}$ can provide nonredundant joint-range inequalities for $H_{\JR}$ when their exposed faces meet $K_{\JR}$ nontrivially.

Define the support function of $S_{\JR}$,
\begin{equation}
  \label{eq:sigma_support}
  \sigma: \bR^2 \to \bR \cup \{\infty\}, \quad \mu \mapsto \sigma(\mu) \doteq \sup_{y\in S_{\JR}}\mu^\top y.
\end{equation}
Standard properties of support functions \cite{hiriart2004fundamentals} include that $\sigma$ is convex and positively 1-homogeneous (hence subadditive), as it is the pointwise supremum of linear functions; its effective domain $\dom(\sigma)\doteq\{\mu:\sigma(\mu)<\infty\}$ is a convex cone; and every valid supporting inequality for $S_{\JR}$ has the form $\mu^\top y \le \sigma(\mu)$.

To evaluate $\sigma(\mu)$ in closed form, define the aggregated coefficients
$Q(\mu)\doteq \mu_1\Theta_1+\mu_2\Theta_2$ and $\varrho(\mu)\doteq \mu_1\theta_1+\mu_2\theta_2$.
Recall that the \emph{Moore--Penrose pseudoinverse} of a symmetric matrix $A$ with eigenvalue decomposition $A = U\Lambda U^\top$ is
\[
  A^\dagger \;\doteq\; U\Lambda^\dagger U^\top, \qquad
  (\Lambda^\dagger)_{ii} =
  \begin{cases}
    \lambda_i^{-1}, & \lambda_i \neq 0, \\
    0,              & \lambda_i = 0.
  \end{cases}
\]

\begin{lemma}
  \label{lem:sigma-formula}
  The support function of $S_{\JR}$ is given by
  \begin{equation}
    \label{eq:sigma_formula}
    \sigma(\mu)=
    \begin{cases}
      +\infty,                                                          & Q(\mu)\not\preceq 0\;\text{or}\;\varrho(\mu)\notin\rang(Q(\mu)), \\[2pt]
      \tfrac{1}{4}\,\varrho(\mu)^\top (-Q(\mu))^\dagger\, \varrho(\mu), & Q(\mu)\preceq 0,\;\varrho(\mu)\in\rang(Q(\mu)).
    \end{cases}
  \end{equation}
  In the finite case $Q(\mu)\preceq 0$ and
  $\varrho(\mu)\in\rang(Q(\mu))$, the supremum is attained at
  $x^\star(\mu)=\tfrac{1}{2}\,(-Q(\mu))^\dagger\,\varrho(\mu)+z$ for any
  $z\in\ker Q(\mu)$.
\end{lemma}
\begin{proof}
  Since $\mu_1 f_1(x)+\mu_2 f_2(x)=x^\top Q(\mu)\,x+\varrho(\mu)^\top x$, we have $\sigma(\mu)=\sup_{x}\bigl(x^\top Q(\mu)\,x+\varrho(\mu)^\top x\bigr)$. Three cases arise:
  \begin{enumerate}[label=(\alph*)]
    \item If $Q(\mu)\not\preceq 0$, the objective is unbounded along any positive-eigenvalue direction, so $\sigma(\mu)=+\infty$.
    \item If $Q(\mu)\preceq 0$ but $\varrho(\mu)\notin\rang(Q(\mu))$, the linear term is unbounded on $\ker Q(\mu)$, so $\sigma(\mu)=+\infty$.
    \item If $Q(\mu)\preceq 0$ and $\varrho(\mu)\in\rang(Q(\mu))$, write the eigenvalue decomposition of the rank-$r$ matrix $Q(\mu)=U_r\Lambda_r U_r^\top$ with negative diagonal entries $\lambda_i<0$ in $\Lambda_r$. Decomposing $x=U_r\xi+z$ ($z\in\ker Q(\mu)$) reduces the objective to the strictly concave quadratic $\xi^\top\Lambda_r\xi+(U_r^\top\varrho(\mu))^\top \xi$. Completing the square yields the maximum $\tfrac{1}{4}\,\varrho(\mu)^\top(-Q(\mu))^\dagger \varrho(\mu)$, attained at $\xi^\star=-\tfrac{1}{2}\Lambda_r^{-1}U_r^\top\varrho(\mu)$.
  \end{enumerate}
\end{proof}

Given a query point $y^*\in\bR^2$, define the \emph{cut violation function}
\begin{equation}
  \label{eq:support_separation}
  v(\mu)\doteq\sigma(\mu)-\mu^\top y^*
  =\sup_{y\in S_{\JR}}\mu^\top(y-y^*).
\end{equation}
A violated valid inequality exists if and only if $\inf_\mu v(\mu)<0$. Since
$v$ is positively 1-homogeneous in $\mu$, we find a bounded convex region
$\bB_2$ (e.g., unit Euclidean ball) and impose a normalization constraint $\mu \in \bB_2$, reducing the separation problem to
\begin{equation}
  \label{eq:separation-problem}
  \inf_{\mu \in \dom(\sigma) \cap \bB_2}\; v(\mu).
\end{equation}

\begin{theorem}[Convex reformulation of separation problem]
  \label{thm:sdp-separation}
  The domain condition $Q(\mu)\preceq 0$, $\varrho(\mu)\in\rang(Q(\mu))$ is equivalent to the existence of $\zeta\in\bR$ satisfying
  \begin{equation}
    \label{eq:sdp-domain}
    \begin{pmatrix}
      -Q(\mu)                       & \tfrac{1}{2}\varrho(\mu) \\
      \tfrac{1}{2}\varrho(\mu)^\top & \zeta
    \end{pmatrix}
    \succeq 0.
  \end{equation}
  Consequently, \eqref{eq:separation-problem} is equivalent to the semidefinite constrained convex program
  \begin{equation}
    \label{eq:sdp-separation}
    \min_{\mu\in \bB_2,\,\zeta}\;\bigl(\zeta - \mu^\top y^*\bigr)
    \quad\text{s.t.}\quad
    \begin{pmatrix}
      -Q(\mu)                       & \tfrac{1}{2}\varrho(\mu) \\
      \tfrac{1}{2}\varrho(\mu)^\top & \zeta
    \end{pmatrix}
    \succeq 0.
  \end{equation}
\end{theorem}
\begin{proof}
  By the generalized Schur complement lemma \cite{albert1969conditions}, the matrix in~\eqref{eq:sdp-domain} is positive semidefinite if and only if $-Q(\mu)\succeq 0$, $\varrho(\mu)\in\rang(Q(\mu))$, and
  \[
    \zeta \ge \tfrac{1}{4}\varrho(\mu)^\top(-Q(\mu))^\dagger\varrho(\mu).
  \]
  By \Cref{lem:sigma-formula}, the right-hand side equals $\sigma(\mu)$ on $\dom(\sigma)$. Minimizing $\zeta$ therefore recovers $\zeta^\star=\sigma(\mu)$, and the objective becomes $v(\mu)$. Since $Q(\mu)$ and $\varrho(\mu)$ are affine in $\mu$, the constraint~\eqref{eq:sdp-domain} is a linear matrix inequality in $(\mu,\zeta)$.
\end{proof}
Thus,  the lifted joint-range inequality \begin{equation}
  \label{eq:lifted-joint-range-inequality}
  \mu^\top\cA_{\JR}(X,x) \le \sigma(\mu)
\end{equation}
is  a linear combination of the affine expressions $\cA_{\JR}(X,x)$  of two base inequalities with a tightened right-hand side.

Recall that the affine map $\cA_{\JR}(X,x)$ defined in  \eqref{eq:affine-map} is expressed as $(\langle \Theta_1, X \rangle + \theta_1^\top x,\;
    \langle \Theta_2, X \rangle + \theta_2^\top x)$. Let $J \subseteq \idxset{n}$ with $n_J = |J|$. We call $J$ a \emph{quadratic support} of $\cA_{\JR}$, if $\Theta_1$ and $\Theta_2$ are zero outside the subblock indexed by $J$, so that nonzero entries appear in the submatrices $(\Theta_1)_{J,J}$ and $(\Theta_2)_{J,J}$. The celebrated Shor's SDP relaxation \cite{shor1987quadratic} yields an extended formulation of the convex joint range, expressed as a spectrahedral shadow.

\begin{corollary}[Extended formulation of the convex joint range]
  \label{cor:sdp-lifted-joint-range}
Assume that $S_{\JR}$ is convex, and let $J$ be a quadratic support of $\cA_{\JR}$. Then,
  \begin{equation}
    \label{eq:sdp-lifted-joint-range}
    \clos(S_{\JR}) =
    \clos\left\{ \cA_{\JR}(X,x) \in \bR^2 :\;
    \begin{pmatrix} X_{J,J} & x_J \\ x_J^\top & 1 \end{pmatrix} \succeq 0,\;
    X \in \bS^n,\;  x \in \bR^n \right\}.
  \end{equation}
\end{corollary}
\begin{proof}
  For ``$\subseteq$'': if $y \in S_{\JR}$, choose $x$ with $F(x) = y$ and set $X_{J,J} = x_Jx_J^\top$; then $\begin{pmatrix} X_{J,J} & x_J \\ x_J^\top & 1 \end{pmatrix} \succeq 0$, so $S_{\JR}$ is contained in the SDP image. Taking closures gives ``$\subseteq$''.
  For ``$\supseteq$'': let $(X,x)$ satisfy $X_{J,J} \succeq x_Jx_J^\top$ and set $y = \cA_{\JR}(X,x)$. Since $S_{\JR}$ is convex, $\clos(S_{\JR}) = \{y : \mu^\top y \le \sigma(\mu),\; \forall\, \mu \in \dom(\sigma)\}$. Take any $\mu \in \dom(\sigma)$. It follows from \Cref{lem:sigma-formula}, $Q(\mu) \preceq 0$ that
  \begin{equation*}
      \begin{split}
           & \langle Q(\mu), X \rangle = \langle Q(\mu)_{J,J}, X_{J,J} \rangle \\
    = & \langle Q(\mu)_{J,J}, x_Jx_J^\top \rangle + \langle Q(\mu)_{J,J}, X_{J,J} - x_Jx_J^\top \rangle
    \le x_J^\top Q(\mu)_{J,J}\, x_J.
      \end{split}
  \end{equation*}
 As $ x_J^\top Q(\mu)_{J,J}\, x_J =  x^\top Q(\mu)\, x$,
  \[
    \mu^\top y
    = \langle Q(\mu), X \rangle + \varrho(\mu)^\top x
    \le x^\top Q(\mu)\, x + \varrho(\mu)^\top x
    \le \sigma(\mu),
  \]
  where the last inequality holds by definition~\eqref{eq:sigma_support}. Since this is valid for every $\mu \in \dom(\sigma)$, we conclude $y \in \clos(S_{\JR})$. Thus the SDP image and its closure are contained in the closed set $\clos(S_{\JR})$, and ``$\supseteq$'' follows.
\end{proof}

\begin{remark}[A sparse outer approximation of Shor's SDP relaxation]
  The joint-range inequalities \eqref{eq:lifted-joint-range-inequality} obtained from the convex joint range are implied by Shor's SDP relaxation of QCQP.
  If the quadratic support $J$ given by the two base inequalities \eqref{eq:base-cut} has a small cardinality $n_J$, then the PSD constraint in \eqref{eq:sdp-lifted-joint-range} involves only a local block of order $n_J+1$. In addition, the lifted joint-range inequality \eqref{eq:lifted-joint-range-inequality} has the same quadratic support $J$ as $\Theta_1$ and $\Theta_2$: it is a linear combination (aggregation) of sparse expressions $\langle \Theta_1, X \rangle + \theta_1^\top x $ and $
    \langle \Theta_2, X \rangle + \theta_2^\top x$ weighted by $\mu$ with the right-hand side $\sigma(\mu)$. Thus, the proposed joint-range inequalities are sparsity-preserving, i.e., the inequalities inherit the sparsity of the base inequalities.
\end{remark}

\section{Secant Mixed-Joint-Range Inequalities}\label{sec:MIR}
This section develops the mixed-joint-range version of the same project-then-lift pipeline. As before, the goal is to derive a cut from a two-dimensional
projected geometry and then lift it back to the extended QCQP formulation. The difference is that we first extract suitable mixed terms from the base inequalities so that the
remaining two-dimensional core pair produces a nonconvex parabolic joint range.

This construction targets the nontrivial nonconvex cases ($\Delta<0$) in \Cref{alg:joint-range}; the punctured nonconvex cases with $\Delta=0$ are not covered by the
parabolic bowl construction below. As with the passage from Chvátal to MIR inequalities, the purpose of the mixed terms is to absorb problematic
quadratic or linear contributions from the base inequalities, leaving a core pair that satisfies the parabolic geometric conditions. The resulting
secant  mixed-joint-range inequality is then obtained by applying the same projected intersection-cut logic to that core geometry as in
\Cref{rem:parabolic-secant-intersection-cuts}.

Given the two base inequalities~\eqref{eq:base-cut},
suppose that we identify two additional affine inequalities in the lifted variables represented by the ``mixed'' slack variables $y_3,y_4$:
\begin{equation}
  y_3 \doteq \phi_3 + \langle \Theta_3, X \rangle + \theta_3^\top x \ge 0, \qquad
  y_4 \doteq \phi_4 + \langle \Theta_4, X \rangle + \theta_4^\top x \ge 0,
\end{equation}
and extract them from the base inequalities~\eqref{eq:base-cut} to define the core terms
\begin{equation}
  \begin{aligned}
    y_1 & \doteq \langle \Theta_1, X \rangle + \theta_1^\top x - a_{13}y_3 - a_{14}y_4, \\
    y_2 & \doteq \langle \Theta_2, X \rangle + \theta_2^\top x - a_{23}y_3 - a_{24}y_4.
  \end{aligned}
\end{equation}
The base inequalities~\eqref{eq:base-cut} are thus equivalent to
\begin{equation}
  \label{eq:mixed-base-rewrite}
  \begin{aligned}
    y_1 + a_{13}y_3 + a_{14}y_4 + \phi_1 & \ge 0, \\
    y_2 + a_{23}y_3 + a_{24}y_4 + \phi_2 & \ge 0.
  \end{aligned}
\end{equation}
We assume that coefficients $a_{13},a_{14},a_{23},a_{24} \ge 0$. (If a coefficient were negative, we could simply drop the term to relax the base inequality, as in the MIR generation heuristic \cite{marchand2001aggregation}).

Let $K_{\MP}$ denote the four-dimensional cone in the $(y_1,y_2,y_3,y_4)$ variables defined by \eqref{eq:mixed-base-rewrite} together with $y_3,y_4\ge0$. Its apex is $y'=(-\phi_1,-\phi_2,0,0)$, and the four defining inequalities yield the four extreme rays used below. In this coordinate system, $K_{\MP}$ is simplicial.

Define the affine mapping $\cA_{\MP}(X,x) \doteq (y_1,y_2,y_3,y_4)$. To generate an intersection cut, we require the core terms to produce a suitable nonconvex geometry in the $(y_1,y_2)$-plane. Letting $\pi_{12}(y) \doteq (y_1,y_2)$ denote the projection onto the first two coordinates, we assume:

\begin{enumerate}[label=(A\arabic*)]
  \item \textbf{Parabolic nonconvex core range:} The extracted core pair defines a two-dimensional joint range
        \[
          S_{\MP} \doteq \bigl\{ \pi_{12}(\cA_{\MP}(xx^\top, x)) : x \in \bR^n \bigr\} \subseteq \bR^2,
        \]
        where $\pi_{12}$ projects the four-dimensional image $\cA_{\MP}(xx^\top,x)$ onto its first two coordinates. We assume that this two-dimensional core range falls into one of the parabolic  ($\Delta<0$) cases in \Cref{alg:joint-range}. Specifically, there exists an invertible
        affine map $T:\bR^2\to\bR^2$ and a closed, full-dimensional convex
        parabolic bowl $C_{\Par}\subseteq\bR^2$ such that, with
        $\widehat C\doteq T(C_{\Par})$, one has
        \[
          S_{\MP}\in\{\bd(\widehat C),\clos(\widehat C^c)\}.
        \]
        Constant terms in the extracted core functions are absorbed into this affine translation before applying the parabolic classification.
  \item \textbf{Interior apex:} The apex $y' \doteq (-\phi_1,-\phi_2,0,0)$ of $K_{\MP}$ projects to the interior of $\widehat C$, i.e., $y'_{12} \doteq (-\phi_1,-\phi_2) \in \inter(\widehat C)$.
\end{enumerate}

These assumptions immediately establish a valid separation framework in $\bR^4$.
\begin{lemma}[Free-set and apex conditions]
  \label{lem:free-set-apex-condition}
  Under (A1) and (A2), the cylinder $C_{\MP} \doteq \widehat C \times \bR^2$ is an $(S_{\MP} \times \bR^2)$-free set, and the apex $y'$ of $K_{\MP}$ lies in $\inter(C_{\MP})$.
\end{lemma}
\begin{proof}
  By (A1), $S_{\MP}$ is either $\bd(\widehat C)$ or $\clos(\widehat C^c)$, so
  $S_{\MP} \cap \inter(\widehat C) = \varnothing$. Thus
  $(S_{\MP} \times \bR^2) \cap \inter(C_{\MP}) = \varnothing$, making $C_{\MP}$ an
  $(S_{\MP} \times \bR^2)$-free set. Since
  $y'_{12} \in \inter(\widehat C)$ by (A2), we have $y' \in \inter(C_{\MP})$.
\end{proof}

We can now construct the \emph{secant  mixed-joint-range inequality}, following the template of deriving MIR inequalities as intersection cuts in \Cref{rem:mir-as-lifted-intersection-cut}. The four inequalities defining $K_{\MP}$ naturally define slack variables $\eta \ge 0$:
\begin{equation}
  \label{eq:mixed-slack}
  \begin{aligned}
    \eta_1 & \doteq y_1 + a_{13}y_3 + a_{14}y_4 + \phi_1, \\
    \eta_2 & \doteq y_2 + a_{23}y_3 + a_{24}y_4 + \phi_2, \\
    \eta_3 & \doteq y_3,                                  \\
    \eta_4 & \doteq y_4.
  \end{aligned}
\end{equation}
The extreme rays $r_j$ of $K_{\MP}$ correspond to setting $\eta_j > 0$ whereas keeping the other slacks at zero. Inverting the slack definition~\eqref{eq:mixed-slack} yields the ray directions in the $y$-space:
\[
  \begin{aligned}
    r_1 & = e_1,                                  & \quad
    r_2 & = e_2,                                          \\
    r_3 & = (-a_{13},\, -a_{23},\, 1,\, 0)^\top,  & \quad
    r_4 & = (-a_{14},\, -a_{24},\, 0,\, 1)^\top .
  \end{aligned}
\]
The step length $\eta^*_j$ along each ray $r_j$ before exiting $C_{\MP}$ depends only on the $(y_1,y_2)$ coordinates, because $C_{\MP} = \widehat C \times \bR^2$. We compute
\begin{equation}
  \label{eq:mixed-step-length}
  \eta^*_j \doteq \sup\bigl\{\eta_j \ge 0 : \pi_{12}(y') + \eta_j\, \pi_{12}(r_j) \in \widehat C\bigr\}.
\end{equation}
For the core rays $j=1,2$, the projected directions are $e_1$ and $e_2$, matching the basic 2D parabolic case. For the mixed rays $j=3,4$, the projected directions are $(-a_{13},-a_{23})^\top$ and $(-a_{14},-a_{24})^\top$.
Then, we can construct the projected intersection cut in the four-dimensional image space:
\begin{equation}
  \label{eq:mixed-intersection-cut}
  \sum_{j=1}^{4} \frac{\eta_j}{\eta^*_j} \ge 1.
\end{equation}
As in the standard intersection-cut convention, we set $1/\eta_j^*=0$ when $\eta_j^*=\infty$.

At $y = \cA_{\MP}(X,x)$, the extracted terms cancel in the first two slack coordinates:
\begin{equation}
  \begin{aligned}
    \eta_1 & = y_1 + a_{13}y_3 + a_{14}y_4 + \phi_1
    = \langle \Theta_1, X \rangle + \theta_1^\top x + \phi_1,                \\
    \eta_2 & = y_2 + a_{23}y_3 + a_{24}y_4 + \phi_2
    = \langle \Theta_2, X \rangle + \theta_2^\top x + \phi_2,                \\
    \eta_3 & = y_3 = \phi_3 + \langle \Theta_3, X \rangle + \theta_3^\top x, \\
    \eta_4 & = y_4 = \phi_4 + \langle \Theta_4, X \rangle + \theta_4^\top x.
  \end{aligned}
\end{equation}
Substituting these slack identities into~\eqref{eq:mixed-intersection-cut} yields the explicit lifted intersection cut valid for QCQP:
\begin{equation}
  \label{eq:mixed-lifted-intersection-cut}
  \begin{split}
     & \frac{\langle \Theta_1, X \rangle + \theta_1^\top x + \phi_1}{\eta^*_1} +
    \frac{\langle \Theta_2, X \rangle + \theta_2^\top x + \phi_2}{\eta^*_2}               \\
     & \qquad + \frac{\phi_3 + \langle \Theta_3, X \rangle + \theta_3^\top x}{\eta^*_3} +
    \frac{\phi_4 + \langle \Theta_4, X \rangle + \theta_4^\top x}{\eta^*_4} \ge 1.
  \end{split}
\end{equation}
We call the lifted intersection cut~\eqref{eq:mixed-lifted-intersection-cut} a secant  mixed-joint-range inequality.

\begin{theorem}[Secant mixed-joint-range inequality]
  \label{thm:mixed-joint-range-secant}
  Under the above construction, with \(y_3,y_4\) valid nonnegative slacks, \(a_{13},a_{14},a_{23},a_{24}\ge0\), and assumptions (A1) and (A2), the projected intersection cut~\eqref{eq:mixed-intersection-cut}
  is a valid inequality for
  \[
    \clos(\conv(K_{\MP} \setminus \inter(C_{\MP}))),
  \]
  and the secant  mixed-joint-range inequality~\eqref{eq:mixed-lifted-intersection-cut} is valid for QCQP.
\end{theorem}
\begin{proof}
  By \Cref{lem:free-set-apex-condition}, $C_{\MP}$ is $(S_{\MP} \times \bR^2)$-free and contains the apex $y'$ in its interior. Applying the standard intersection cut construction to the cone $K_{\MP} = \{y : \eta(y) \ge 0\}$ and the set $C_{\MP}$ establishes the validity of \eqref{eq:mixed-intersection-cut} for $\clos(\conv(K_{\MP} \setminus \inter(C_{\MP})))$.

  For any feasible QCQP solution, its lifted point $(X,x)$ satisfies $X_{\cE}=(xx^\top)_{\cE}$, so the mapping $\cA_{\MP}$ ensures $(y_1,y_2) \in S_{\MP}$ and therefore $\cA_{\MP}(X,x) \in S_{\MP} \times \bR^2$. Furthermore, the inequalities defining $K_{\MP}$ are valid for QCQP, so $\cA_{\MP}(X,x) \in K_{\MP}$. Thus, $\cA_{\MP}(X,x) \in (S_{\MP} \times \bR^2) \cap K_{\MP}$. Applying \Cref{thm:lifted-intersection-cut} gives the pullback validity of~\eqref{eq:mixed-lifted-intersection-cut}.
\end{proof}

\begin{remark}[Generalization]
  \label{rem:mixed-generalization}
  The construction of secant mixed-joint-range  inequalities is not limited to two mixed slack variables. It extends directly to $k-2$ additional valid inequalities by introducing slack variables $y_3,\dots,y_k$, extracting nonnegative combinations of these slacks from the two base inequalities, and constructing the intersection cut in the resulting $k$-dimensional image space. The lifted inequality is valid for QCQPs by the same argument as above. One may also introduce only one additional slack variable
  \[
    y_3 \doteq \phi_3+\langle\Theta_3,X\rangle+\theta_3^\top x\ge 0
  \]
  and choose nonnegative multipliers $a_{13},a_{23}\ge0$ to define the residual core pair
  \[
    \begin{aligned}
      y_1 & \doteq \langle \Theta_1-a_{13}\Theta_3,X\rangle
      +(\theta_1-a_{13}\theta_3)^\top x
      -a_{13}\phi_3,                                        \\
      y_2 & \doteq \langle \Theta_2-a_{23}\Theta_3,X\rangle
      +(\theta_2-a_{23}\theta_3)^\top x
      -a_{23}\phi_3.
    \end{aligned}
  \]
  Equivalently, when $X_{\cE}=(xx^\top)_{\cE}$, the residual quadratic functions corresponding to $y_i$ are
  \[
    x^\top \Theta_i^{\rm res}x+(\theta_i^{\rm res})^\top x-a_{i3}\phi_3,
    \qquad
    \Theta_i^{\rm res}\doteq \Theta_i-a_{i3}\Theta_3,\quad
    \theta_i^{\rm res}\doteq \theta_i-a_{i3}\theta_3,\quad i=1,2.
  \]
  The constants \(-a_{i3}\phi_3\) translate the residual two-dimensional image and are absorbed into the affine offset before applying the parabolic classification.
  The slack $y_3$ is selected so that the residual quadratic parts $\Theta_1^{\rm res}$ and $\Theta_2^{\rm res}$ are linearly dependent and the residual linear parts $\theta_1^{\rm res},\theta_2^{\rm res}$ satisfy the parabolic classification conditions in Algorithm~\ref{alg:joint-range}.

\end{remark}

\section{Numerical Experiments}\label{sec:experimental}

A systematic study of  joint-range  inequalities inside a full MINLP solver is beyond the scope of this paper. Instead, we report preliminary geometric experiments that use projected area as a proxy for relaxation strength~\cite{lee2007mixed,lee1994geometric,speakman2017quantifying}.

We use the notation introduced in \eqref{eq:base-cut}, \eqref{eq:affine-map}, \eqref{eq.quadmap}, and \eqref{eq:KJR}. Thus $F$ is the quadratic map in \eqref{eq.quadmap}, $y=\cA_{\JR}(X,x)$ is the affine image in \eqref{eq:affine-map}, and the two base inequalities \eqref{eq:base-cut} are represented in the image space by the simplicial cone $K_{\JR}$ in \eqref{eq:KJR}. The standard first-level box RLT relaxation for $X_{\cE}=(xx^\top)_{\cE}$ on $[0,1]^n$ is given by:
\begin{align*}
  0 & \le x_i \le 1
    &                    & \forall i = 1,\ldots,n,
  \\
  0 & \le X_{ij} \le x_i
    &                    & \forall 1 \le i \le j \le n,
  \\
  x_i + x_j - 1
    & \le X_{ij} \le x_j
    &                    & \forall 1 \le i \le j \le n.
\end{align*}
We denote this polyhedron by $\mathcal R_{\rm RLT}$. Its projection to the two-dimensional image space, restricted by the base inequalities, is
\[
  H_{\rm RLT}
  = \left\{y\in\bR^2:
  \begin{array}{l}
    \exists (X,x)\in\mathcal R_{\rm RLT}, \\
    y=\cA_{\JR}(X,x),                    \\
    y\in K_{\JR}
  \end{array}
  \right\}.
\]
Equivalently, the last condition is $\phi_1+y_1\ge0$ and $\phi_2+y_2\ge0$.
The experiments evaluate how much projected RLT area remains after adding the joint-range hull description induced by the two base inequalities. We therefore compare the projected RLT relaxation $H_{\rm RLT}$ with the joint-range hull $H_{\JR} \doteq \clos(\conv(F(\bR^n)\cap K_{\JR}))$ through their overlap:
\[
  \rho =
  \frac{\operatorname{area}(H_{\JR}\cap H_{\rm RLT})}
  {\operatorname{area}(H_{\rm RLT})}.
\]
Thus $\rho=1$ means that intersecting with the joint-range hull leaves the projected RLT area unchanged, whereas smaller values indicate a smaller remaining overlap.

\subsection{Experiment Setup}

All experiments were run on a Linux machine with an AMD Ryzen 7 8845H processor and 32 GB of RAM. The implementation used Python 3.11, Shapely for area computations, CVXPY~\cite{diamond2016cvxpy} with SCS~\cite{o2016conic} for SDPs, and SCIP~\cite{hojny2025scip} through PySCIPOpt~\cite{maher2016pyscipopt} for LPs.

\subsection{Geometric Cases}

The columns of Table~\ref{tab:mccormick-ratio} correspond to the geometric cases of the joint-range hull $H_{\JR}$ studied in the previous sections. Since the theoretical characterizations of the convex hull in the nonconvex cases are derived in the canonical coordinate system (denoted by tildes), in
  Cases~1--6 we express the joint-range hull as $H_{\JR} = D(\relx{H}_{\JR} + c)$, where $\relx{H}_{\JR}$ takes one of the following canonical forms; the
convex Case~7 is stated directly in the original coordinates. The first two cases come from the containment configuration $C_{\Par}\subseteq \relx{K}_{\JR}$, where $\relx{K}_{\JR}$ is the transformed simplicial cone, in \Cref{thm:conv-hull-contained-bowl}; the next four come from the interior-apex configuration $q(\relx{y}')<0$ in \Cref{thm:conv-hull-interior-apex}; the last case is the convex joint-range case, in which the joint-range hull equals the intersection of the spectrahedral shadow of \Cref{cor:sdp-lifted-joint-range} with the cone $K_{\JR}$.

\begin{itemize}
  \item \textbf{Case 1, Unclipped Parabolic Bowl}: $C_{\Par}\subseteq \relx{K}_{\JR}$ and $\relx{S}_{\JR}=\bd(C_{\Par})$. Then
        \[
          \relx{H}_{\JR} \;= \mathcal{P}_{\text{full}} \doteq  C_{\Par}.
        \]
  \item \textbf{Case 2, Unclipped Polyhedral Cone}: $C_{\Par}\subseteq \relx{K}_{\JR}$ and $\relx{S}_{\JR}=\clos(C_{\Par}^c)$. Then
        \[
          \relx{H}_{\JR}  =   \mathcal{K}_{\text{full}} \doteq \relx{K}_{\JR}.
        \]
  \item \textbf{Case 3, Chord-Truncated Bowl}: $q(\relx{y}')<0$ and both boundary rays of $\relx{K}_{\JR}$ meet $\bd(C_{\Par})$ at points $\relx{v}_1,\relx{v}_2$, with $\relx{S}_{\JR}=\bd(C_{\Par})$. Then
        \[
          \relx{H}_{\JR}  =    \mathcal{P}_{\text{trunc}} \doteq  C_{\Par}\cap H_{12}.
        \]
  \item \textbf{Case 4, Chord-Truncated Cone}: same configuration as the previous case but with $\relx{S}_{\JR}=\clos(C_{\Par}^c)$. Then
        \[
          \relx{H}_{\JR}  =    \mathcal{K}_{\text{trunc}} \doteq  \relx{K}_{\JR}\cap H_{12}.
        \]
  \item \textbf{Case 5, Ray-Truncated Bowl}: $q(\relx{y}')<0$, one boundary ray is a recession ray contained in $C_{\Par}$, the other meets $\bd(C_{\Par})$ at a unique point $\relx{v}$, with $\relx{S}_{\JR}=\bd(C_{\Par})$. Then
        \[
          \relx{H}_{\JR}  =   \mathcal{P}_{\text{ray}} \doteq  C_{\Par}\cap H_{\relx{v}}.
        \]
  \item \textbf{Case 6, Ray-Truncated Cone}: same configuration as the previous case but with $\relx{S}_{\JR}=\clos(C_{\Par}^c)$. Then
        \[
          \relx{H}_{\JR}  =     \mathcal{K}_{\text{ray}} \doteq \relx{K}_{\JR}\cap H_{\relx{v}}.
        \]
   \item \textbf{Case 7, Convex Spectrahedral Shadow}: $S_{\JR}$ is convex, and its closure is described by the SDP image in \Cref{cor:sdp-lifted-joint-range}. Let the intersection of the SDP image with $K_{\JR}$ be $\mathcal{S}_{\text{shadow}}$. We have that
        $
            \clos(\mathcal{S}_{\text{shadow}}) = \clos(H_{\JR}).$
\end{itemize}

\subsection{Data Generation and Area Computation}

For each $n\in\{3,4,5,6,7,8\}$ and each case, we sample ten accepted instances. The nonconvex Cases~1--6 are generated from the canonical parabolic forms in Appendix~\ref{sec:app-canonical-transformation}; Case~7 is generated separately from linearly independent positive definite forms. The random generation scheme and numerical approximation parameters are summarized in Appendix~\ref{sec:app-experiment-generation}.

\medskip
\noindent\textbf{Numerical evaluation.}
The projected RLT set $H_{\rm RLT}$ and, in Case~7, the SDP representation of $\clos(S_{\JR})$ together with the cone constraints, are approximated by support-function sampling: for each direction in the $(y_1,y_2)$ plane, an LP over $\mathcal{R}_{\rm RLT}$ (or an SDP) is solved, and the resulting support points are convexified. Polygon intersections and areas are computed with Shapely.

\subsection{Results and Analysis}

For each pair of dimension and case, we report the shifted geometric mean over ten accepted random instances,
\[
  \exp\left(\frac{1}{10}\sum_{k=1}^{10}\log(1+\rho_k)\right)-1.
\]
The shift by one makes the statistic well defined when some sampled ratios are zero.
All random sampling uses fixed base seed $1$; each accepted instance records the deterministic attempt seed used to initialize its generator.

\begin{table}[!htbp]
  \centering
  \caption{Shifted geometric mean of the area ratio $\rho$ over ten random instances for each dimension and geometric case.}
  \label{tab:mccormick-ratio}
  \begin{tabular}{rrrrrrrr}
\toprule
$n$ & 1 & 2 & 3 & 4 & 5 & 6 & 7 \\
\midrule
3 & 0.7581 & 1.0000 & 0.4895 & 0.7168 & 0.3922 & 0.7841 & 0.3298 \\
4 & 0.8503 & 1.0000 & 0.4259 & 0.6636 & 0.4211 & 0.6014 & 0.2821 \\
5 & 0.8479 & 1.0000 & 0.1697 & 0.7036 & 0.2684 & 0.2919 & 0.2185 \\
6 & 0.8094 & 1.0000 & 0.3306 & 0.8343 & 0.3047 & 0.3364 & 0.2335 \\
7 & 0.8286 & 1.0000 & 0.1800 & 0.6913 & 0.3522 & 0.1739 & 0.4031 \\
8 & 0.8434 & 1.0000 & 0.2038 & 0.7345 & 0.2492 & 0.1082 & 0.2206 \\
\bottomrule
\end{tabular}

\end{table}

Table~\ref{tab:mccormick-ratio} reveals three main patterns. First, Case~2 (Unclipped Polyhedral Cone) always has ratio one, as expected: in this containment regime the joint-range hull coincides with the simplicial cone already enforced in $H_{\rm RLT}$. Second, Case~1 (Unclipped Parabolic Bowl) gives ratios between $0.758$ and $0.850$, corresponding to about $15\%$--$24\%$ projected area removal. Third, the interior-apex and convex joint-range cases are substantially stronger: excluding the trivial Case~2, the shifted-geometric-mean ratios range from $0.108$ to $0.850$, with mean $0.474$ across the case--dimension pairs.

The smallest average overlap occurs in the interior-apex boundary cases and the convex case. Case~3 has ratios $0.170$--$0.490$, Case~5 has ratios $0.249$--$0.421$, and Case~7 has ratios $0.219$--$0.403$. The solid interior-apex cases are more orientation-sensitive: Case~4 remains moderate, with ratios $0.664$--$0.834$, whereas Case~6 ranges from $0.108$ to $0.784$. These computations are approximate because both the RLT projection and, in Case~7, the SDP projection are represented by finitely many support samples, but they support the geometric conclusion that adding the joint-range hull can substantially reduce the projected RLT overlap.

\section{Conclusion and Future Work}\label{sec:conclusion}

We proposed a project-then-lift approach for deriving sparse cutting planes for QCQPs from pairs of base valid inequalities. The derivation procedure starts from two base inequalities, maps them into a two-dimensional image space, characterizes the projected hull generated by the corresponding joint-range geometry, and then lifts the resulting inequalities back to the original formulation. Within this procedure, we introduced joint-range inequalities and secant mixed-joint-range inequalities as concrete cut families for QCQP. In particular, the secant operation is analogous to the simple rounding step used in deriving MIR inequalities, which can also be obtained through the project-then-lift procedure.

The main benefit of the project-then-lift approach  is that it separates geometric strength from ambient dimension. The hard geometric step is handled in two dimensions, where the projected hull admits closed-form descriptions in the nonconvex parabolic cases and support-function or SDP descriptions in the convex case, whereas the final inequalities remain sparse because their support is inherited from the selected base inequalities. For the nonconvex joint range, no SDP relaxation is required; for the convex joint range, this also yields a sparse route to linearize Shor's SDP relaxation; moreover, in the nonconvex case, the framework exposes geometric structure that is inaccessible to approaches based only on hidden convexity.

A primary direction for future research is to automate the first step of the framework: selecting two effective base inequalities for a general QCQP instance. This could be approached through aggregation-based separation heuristics akin to those developed for MIR inequalities~\cite{marchand2001aggregation} or recent studies on sparsity-driven aggregation of MILPs ~\cite{xu2025sparsity} and aggregation of QCQPs~\cite{dey2022obtaining}, but our results suggest that the deeper question is not aggregation alone, but rather which pair of inequalities produces the most useful two-dimensional projected geometry. One could expose convex geometry \cite{dey2022obtaining} and find a viable linear outer approximation of sparse SDP relaxation. To expose the underlying nonconvex geometry, one must ``align'' the coefficients of two quadratic functions so that quadratic parts become linearly dependent, together with satisfying several additional conditions. While such conditions rarely arise naturally, they can often be induced through carefully designed heuristics. A similar phenomenon occurs in the generation of MIR inequalities: the derivation formally requires integral coefficients (see \Cref{lem:alternative-mir}), a pattern that is uncommon in general MILPs. Nevertheless, aggregation heuristics provide an elegant and highly effective mechanism for uncovering such uncommon patterns, for example through coefficient scaling and  quantization. We refer the reader to the highly optimized heuristics implementations of MIR separation developed in, e.g., \cite{christophel2009separation,gonccalves2005implementation,wolter2006implementation}, which demonstrate the significant impact of MIR inequalities and the project-then-lift approach \cite{marchand2001aggregation} on solving general MILPs.

\ifthenelse{\boolean{snjnl}}{\backmatter}{}



\section*{Statements and Declarations}

\begin{itemize}
  \item \textbf{Funding}: This research was partially supported by Research Campus MODAL, funded by the Federal
        Ministry of Research, Technology and Space (BMFTR) (fund numbers 05M14ZAM, 05M20ZBM).
  \item \textbf{Competing Interests}: The authors declare no competing interests.
  \item \textbf{Use of AI tools}: AI tools were used for coding assistance, numerical verification of theorem statements, and Lean-based verification of proofs. The authors reviewed and validated all mathematical statements, proofs, code, computations, and final manuscript text, and take full responsibility for the content of this work.
  \item \textbf{Data availability}: Data and source code are available on \url{https://github.com/lidingxu/Projected-Parabolic-Hull-Area-Computation.git}.
\end{itemize}

\ifthenelse{\boolean{snjnl}}{}{%
  \bibliographystyle{plain}%
}
\bibliography{sn-bibliography}

\appendix

\section{Joint Range of Two Quadratic Functions}
\label{appendix.a}

Consider a pair of inhomogeneous quadratic functions $f_1, f_2 : \bR^n \to \bR$ defined by
\begin{align*}
  f_1(x) & = x^\top \Theta_1 x + \theta_1^\top x, \\
  f_2(x) & = x^\top \Theta_2 x + \theta_2^\top x,
\end{align*}
where $\Theta_1, \Theta_2$ are real symmetric matrices of order $n$ and $\theta_1, \theta_2 \in \bR^n$. We denote the quadratic map by $F : \bR^n \to \bR^2$, $F(x) \doteq (f_1(x), f_2(x))^\top$. This map decomposes into its homogeneous and linear parts:
\[
  F(x) = F_H(x) + F_L(x),
\]
where $F_H(x) = (x^\top \Theta_1 x, x^\top \Theta_2 x)^\top$ is the homogeneous quadratic part and $F_L(x) = (\theta_1^\top x, \theta_2^\top x)^\top$ is the linear part.

We are interested in the joint range
\[
  F(\bR^n) = \{F(x) \in \bR^2 : x \in \bR^n\}.
\]
The homogeneous range $F_H(\bR^n)$ is always a convex cone by Dines' theorem, stated below as \Cref{thm:fbop-dines-homogeneous}; however, the inhomogeneous range $F(\bR^n)$ may be nonconvex.
Some sufficient conditions for convexity are known (e.g., Polyak's condition \cite{polyak1998convexity}), but a complete characterization of nonconvex shapes is available in ~\cite{flores2016characterizing}.

\subsection{Source Results Used in Appendix A}

We state the external results used below in the notation of this appendix. In the notation of Flores-Baz\'an and Opazo~\cite{flores2016characterizing}, $A=\Theta_1$, $B=\Theta_2$, $a=\theta_1$, $b=\theta_2$, and $F=(f_1,f_2)$.

\begin{theorem}[Homogeneous range theorem, \cite{flores2016characterizing}, Thm.~2.1]
  \label{thm:fbop-dines-homogeneous}
  The homogeneous joint range
  \[
    F_H(\bR^n)=\{(x^\top\Theta_1x,\ x^\top\Theta_2x)^\top:x\in\bR^n\}
  \]
  is a convex cone.
\end{theorem}

\begin{theorem}[Fixed-direction criterion, \cite{flores2016characterizing}, Thm.~4.14]
  \label{thm:fbop-fixed-direction}
  Let $d=(d_1,d_2)\in\bR^2\setminus\{0\}$. Then $F(\bR^n)+\bR_+d$ is nonconvex if and only if all of the following conditions hold:
  \begin{align*}
     & \textup{(b1) } \{\theta_1,\theta_2\}\subseteq(\ker\Theta_1\cap\ker\Theta_2)^\perp
    \quad(\Leftrightarrow\ F_L(\ker\Theta_1\cap\ker\Theta_2)=\{0\}),                   \\
     & \textup{(b2) } d_2\Theta_1=d_1\Theta_2,                                          \\
     & \textup{(b3) } -d\in F_H(\bR^n),                                                 \\
     & \textup{(b4) } \text{for every }u\in F_H^{-1}(-d),\;                d_1\theta_2^\top u\neq d_2\theta_1^\top u.
  \end{align*}
\end{theorem}

\begin{theorem}[Converse ray criterion, \cite{flores2016characterizing}, Thm.~4.15]
  \label{thm:fbop-ray-converse}
  If $F(\bR^n)+\bR_+d$ is convex for every $d\in\bR^2\setminus\{0\}$, then $F(\bR^n)$ is convex. Equivalently, if $F(\bR^n)$ is nonconvex, then there exists $d\in\bR^2\setminus\{0\}$ such that $F(\bR^n)+\bR_+d$ is nonconvex.
\end{theorem}

\begin{theorem}[Global convexity characterization, \cite{flores2016characterizing}, Thm.~4.16]
  \label{thm:fbop-global-convexity}
  The joint range $F(\bR^n)$ is convex if and only if, for every $d=(d_1,d_2)\in\bR^2\setminus\{0\}$, at least one of the following conditions holds:
  \begin{itemize}
    \item[(C1)] $F_L(\ker\Theta_1\cap\ker\Theta_2)\neq\{0\}$
          $(\Leftrightarrow\{\theta_1,\theta_2\}\not\subseteq
            (\ker\Theta_1\cap\ker\Theta_2)^\perp)$;
    \item[(C2)] $d_1\Theta_2\neq d_2\Theta_1$;
    \item[(C3)] $-d\notin F_H(\bR^n)$;
    \item[(C4)] there exists $u\in F_H^{-1}(-d)$ such that
          $d_1\theta_2^\top u=d_2\theta_1^\top u$.
  \end{itemize}
\end{theorem}


By \Cref{thm:fbop-dines-homogeneous}, the homogeneous range $F_H(\bR^n)$ is convex. The global result \Cref{thm:fbop-global-convexity} provides a complete characterization of the convex cases, and therefore, by complement, of the nonconvex cases.

\begin{lemma}
  \label{char.nonconvexity}
  Let $f_1,f_2$ be quadratic functions as above. The joint range $
    F(\bR^n)
  $
  is nonconvex if and only if there exists at least one direction $d \in \bR^2$, $d \neq 0$ such that:
  \begin{align*}
     & \text{(b1) } \{\theta_1,\theta_2\}\subseteq (\ker \Theta_1\cap \ker \Theta_2)^\perp
    \quad (\Leftrightarrow\ F_L(\ker \Theta_1\cap \ker \Theta_2)=\{0\}),                   \\
     & \text{(b2) } d_2\Theta_1=d_1\Theta_2,                                               \\
     & \text{(b3) } -d\in F_H(\bR^n),                                                      \\
     & \text{(b4) } \text{for every } u\in F_H^{-1}(-d)
    \text{ (equivalently, } u^\top \Theta_1 u=-d_1,\ u^\top \Theta_2 u=-d_2\text{)},       \\
     & \hspace{3.2em} d_1\theta_2^\top u\neq d_2\theta_1^\top u.
  \end{align*}
\end{lemma}

\begin{proof}
  By \Cref{thm:fbop-global-convexity}, $F(\bR^n)$ is convex if and only if, for every $d\in\bR^2\setminus\{0\}$, at least one of \textup{(C1)--(C4)} holds.
  Hence $F(\bR^n)$ is nonconvex if and only if there exists $d\neq0$ for which
  none of \textup{(C1)--(C4)} holds. The negations of \textup{(C1)--(C4)} are
  exactly \textup{(b1)--(b4)} above. For a fixed direction $d$, these same
  conditions are the fixed-direction criterion in \Cref{thm:fbop-fixed-direction};
  \Cref{thm:fbop-ray-converse} gives the converse existence of such a direction
  when $F(\bR^n)$ itself is nonconvex.
\end{proof}

\begin{corollary}
  \label{cor:LD-necessary}
  If $\Theta_1$ and $\Theta_2$ are linearly independent, then $F(\bR^n)$ is
  convex.
\end{corollary}
\begin{proof}
  Linear independence of $\Theta_1$ and $\Theta_2$ means there is no nonzero
  $d\in\bR^2$ satisfying $d_2\Theta_1 = d_1\Theta_2$, so condition~(C2) of
  \Cref{thm:fbop-global-convexity} holds for every $d\neq 0$.
  By \Cref{thm:fbop-global-convexity}, $F(\bR^n)$ is convex.
\end{proof}

Hence it suffices to consider the case where $\Theta_1$ and $\Theta_2$ are
linearly dependent; if they are linearly independent, the joint range is
already convex.

\subsection{Detailed Canonical Transformation}
\label{sec:app-canonical-transformation}

The procedure checks the candidate directions $d$ for which the fixed-direction
criterion of \Cref{thm:fbop-fixed-direction} can hold. For such
a direction, conditions \textup{(b1)--(b4)} certify that
$F(\bR^n)+\bR_+d$ is nonconvex; consequently $F(\bR^n)$ itself cannot be convex.

The procedure begins with the assumption that condition~(b2) holds, which requires $d = (d_1,d_2)$ to
satisfy
$d_1\Theta_2 = d_2\Theta_1$.
If \(\Theta_1=\Theta_2=0\), then \(F_H(\bR^n)=\{0\}\) and the joint range is affine, hence convex. We therefore assume \((\Theta_1,\Theta_2)\neq(0,0)\).
Since only the direction of $d$ matters, we normalize $\|d\|_2 = 1$; the
solution is then unique up to sign. Therefore, there are two candidate
directions $d$ to check. Given a candidate direction $d$, it is easy to verify
condition~\textup{(b1)} via standard linear algebra routines. Thus, we assume
\textup{(b1)} holds as well. We then check the remaining conditions
\textup{(b3)} and \textup{(b4)} for this fixed candidate direction.

\subsection{Step 1: Simultaneous Diagonalization}

Let $l \doteq \dim(\ker \Theta_1 \cap \ker \Theta_2)$. Let $d_\perp \doteq (-d_2,d_1)$ be the unit vector orthogonal to $d$. Then
\[
  F_H(x) = F_H(x)^\top d\, d + F_H(x)^\top d_\perp\, d_\perp \in \bR^2.
\]

Given the assumption that condition (b2) holds, we have for all $x$,
\begin{equation*}
  F_H(x)^\top d_\perp = x^\top(d_1\Theta_2-d_2\Theta_1)x = 0.
\end{equation*}
Hence $F_H(x)$ is collinear with $d$, so there exists a scalar quadratic function $q_d$ satisfying
\begin{equation}
  \label{eq:app-FH-Q}
  F_H(x) = q_d(x)\,d,
\end{equation}
where
\[
  \bar Q_d\doteq d_1\Theta_1+d_2\Theta_2,
  \qquad
  q_d(x)\doteq F_H(x)^\top d=x^\top\bar Q_d x.
\]

\begin{proposition}
  \label{prop:Q-cases}
  Under condition \textup{(b2)}, let $(m_+,m_-,m_0)$ be the inertia of
  $\bar Q_d \doteq d_1\Theta_1+d_2\Theta_2$ restricted to the active subspace
  $(\ker\Theta_1\cap\ker\Theta_2)^\perp$. Then $m_0=0$. In addition,
  \begin{enumerate}[label=\textup{(\roman*)}]
    \item condition \textup{(b3)} holds if and only if $m_-\ge 1$;
    \item if condition \textup{(b4)} holds, then $m_-\le 1$.
  \end{enumerate}
  In particular, if conditions \textup{(b1)--(b4)} all hold, then $m_-=1$.
\end{proposition}
\begin{proof}
  Condition~(b2) ($d_2\Theta_1=d_1\Theta_2$), together with $\|d\|_2=1$,
  implies $\ker \bar Q_d=\ker\Theta_1\cap\ker\Theta_2$. Thus the restriction of $\bar Q_d$
  to $(\ker\Theta_1\cap\ker\Theta_2)^\perp$ has no zero eigenvalues, i.e.,
  $m_0=0$. Moreover, condition~\textup{(b2)} and $\|d\|_2=1$ imply
  $\Theta_i=d_i\bar Q_d$ for $i=1,2$. Hence diagonalizing
  $\bar Q_d$ on the active subspace and completing by a basis of the common kernel
  gives a nonsingular $\Sigma$ satisfying
  $\Sigma^\top\Theta_i\Sigma
    =d_i\operatorname{diag}(I_{m_+},-I_{m_-},0_l)$ for $i=1,2$.

  For part~\textup{(i)}, condition~\textup{(b2)} gives $F_H(x)=q_d(x)d$.
  Since $\|d\|_2=1$, condition~\textup{(b3)} is equivalent to the existence of
  some $x$ with $q_d(x)=-1$. By homogeneity of $q_d$, this is equivalent to the
  existence of some $x$ with $q_d(x)<0$, which holds if and only if $\bar Q_d$
  has a negative eigenvalue on the active subspace, i.e., $m_-\ge1$.

  For part~\textup{(ii)}, suppose condition~\textup{(b4)} holds and assume
  for contradiction that $m_-\ge 2$.
  Define $\gamma\doteq d_1\theta_2-d_2\theta_1\in\bR^n$ and let
  $\hat{\gamma}\doteq\Sigma^\top\gamma
    =(\gamma_+,\gamma_-,\gamma_0)
    \in\bR^{m_+}\times\bR^{m_-}\times\bR^l$.
  Since $m_-\ge2$, the subspace
  \[
    \{v_-\in\bR^{m_-}:\gamma_-^\top v_-=0\}
  \]
  has dimension at least $m_- -1\ge1$ (and dimension $m_-$ if
  $\gamma_-=0$); choose a unit vector $v_-$ in it.
  Set $\hat{v}\doteq(0,v_-,0)\in\bR^{m_+}\times\bR^{m_-}\times\bR^l$ and
  $u\doteq\Sigma\hat{v}$.
  The block-diagonal structure of $\Sigma^\top\Theta_i\Sigma$ gives,
  for $i=1,2$,
  \[
    u^\top\Theta_i u
    = \hat{v}^\top(\Sigma^\top\Theta_i\Sigma)\hat{v}
    = d_i\bigl(0-\norm{v_-}^2\bigr)
    = -d_i,
  \]
  so $u$ satisfies the antecedent of condition~\textup{(b4)}.
  However,
  \[
    \gamma^\top u
    = \hat{\gamma}^\top\hat{v}
    = \gamma_+^\top 0 + \gamma_-^\top v_- + \gamma_0^\top 0
    = 0,
  \]
  i.e., $d_1\theta_2^\top u=d_2\theta_1^\top u$, contradicting the
  consequent of condition~\textup{(b4)}.
  Hence $m_-\le 1$.
\end{proof}

We assume that condition~\textup{(b3)} holds, i.e., $m_-\ge1$, otherwise the
direction $d$ cannot produce a nonconvex shape. The preceding proposition shows
that condition~\textup{(b4)} cannot hold when $m_-\ge2$. Hence, if $m_-\ne1$,
this candidate direction is discarded. In the remaining case $m_-=1$, we use
the simultaneous diagonalization from the proof as the first change of
variables. With $l=\dim(\ker\Theta_1\cap\ker\Theta_2)$, fix a nonsingular
$\Sigma\in\bR^{n\times n}$ such that
\[
  \Sigma^\top \Theta_1 \Sigma = d_1\operatorname{diag}(I_{m_+},-I_{m_-},0_l),
  \qquad
  \Sigma^\top \Theta_2 \Sigma = d_2\operatorname{diag}(I_{m_+},-I_{m_-},0_l).
\]
Equivalently,
\[
  \Sigma^\top \bar Q_d\Sigma
  = \operatorname{diag}(I_{m_+},-I_{m_-},0_l).
\]
Taking $U=\Sigma$ and using $m_-=1$ gives
\begin{equation}
  \label{eq:app-Q-diag}
  q_d(Uw) = w^\top U^\top \bar Q_d U w = \sum_{i=1}^{m_+} w_i^2 - w_{1 + m_+}^2,
\end{equation}
where $1 + m_+ = n-l$ and $m_+ \ge 0$.

\subsection{Step 2: Change of coordinates and variables}

Combining \eqref{eq:app-FH-Q} and \eqref{eq:app-Q-diag} yields
\begin{equation}
  \label{eq:app-F-diag}
  F(Uw) = \left( \sum_{i=1}^{m_+} w_i^2 - w_{1 + m_+}^2 \right) d + F_L(Uw).
\end{equation}

Since $U$ is nonsingular, the map $w \mapsto Uw$ is a bijection on $\bR^n$, so
$F(\bR^n)= F(U\bR^n)$. Therefore,
$F(\bR^n)=\{F(Uw)^\top d\, d + F(Uw)^\top d_\perp\, d_\perp:w\in\bR^n\}$.

Let $u_i\doteq Ue_i$ and
\[
  \alpha_i \doteq (d_1\theta_1+d_2\theta_2)^\top u_i,
  \qquad
  \beta_i  \doteq (d_1\theta_2-d_2\theta_1)^\top u_i,
  \qquad i=1,\dots,n.
\]
\begin{proposition}
  \label{prop:b1-active}
  Under the  condition \textup{(b1)}, we have $\alpha_i=\beta_i=0$ for all $i>1 + m_+$.
\end{proposition}
\begin{proof}
  Columns $u_{m_+ + 2},\dots,u_n$ span $\ker\Theta_1\cap\ker\Theta_2$, whereas
  condition~(b1) requires $\theta_k\in(\ker\Theta_1\cap\ker\Theta_2)^\perp$,
  so $\theta_k^\top u_i=0$ for $i>1 + m_+$ and $k=1,2$.
  Hence $\alpha_i=(d_1\theta_1+d_2\theta_2)^\top u_i=0$ and
  $\beta_i=(d_1\theta_2-d_2\theta_1)^\top u_i=0$ for all $i>1 + m_+$.
\end{proof}

From this point on, we assume condition~(b1), because otherwise the
direction $d$ cannot produce a nonconvex shape. Combining
\eqref{eq:app-F-diag} with Proposition~\ref{prop:b1-active}, we obtain
\begin{align}
  \label{eq:app-u-step3}
  F(Uw)^\top d       & = \sum_{i=1}^{m_+} w_i^2 - w_{1 + m_+}^2 + \sum_{i=1}^{1 + m_+}\alpha_i w_i, \\
  \label{eq:app-v-step3}
  F(Uw)^\top d_\perp & = \sum_{i=1}^{1 + m_+}\beta_i w_i.
\end{align}
Then \eqref{eq:app-u-step3}--\eqref{eq:app-v-step3} show that the linear
terms involve only the active variables $w_1,\dots,w_{1 + m_+}$.

Next, apply a translation that removes the linear terms from $y_1$.
Define $\Psi:\bR^n \to \bR^n$ by
\[
  (\Psi(v))_i = v_i - \frac{\alpha_i}{2} \quad \text{for } i \le m_+,
  \qquad
  (\Psi(v))_{1 + m_+} = v_{1 + m_+} + \frac{\alpha_{1 + m_+}}{2},
\]
and $(\Psi(v))_i = v_i$ for $i > 1 + m_+$. Since $\Psi$ is a translation, it is a
bijection on $\bR^n$. Reparametrizing with $w=\Psi(v)$ therefore leaves
the image unchanged:
\[
  F(\bR^n)
  =
  \left\{
  F(U\Psi(v))^\top d\, d
  +
  F(U\Psi(v))^\top d_\perp\, d_\perp
  :\; v \in \bR^n
  \right\}.
\]
A direct substitution gives
\begin{align*}
  F(U\Psi(v))^\top d
   & = \sum_{i=1}^{m_+} v_i^2 - v_{1 + m_+}^2
  + c_d,                                      \\
  F(U\Psi(v))^\top d_\perp
   & = \sum_{i=1}^{1 + m_+} \beta_i v_i
  + c_{d_\perp},
\end{align*}
where
\begin{align}
  \label{eq.coeff-c}
  c_d         & = \frac{1}{4}\Big(\alpha_{1 + m_+}^2 - \sum_{i=1}^{m_+} \alpha_i^2\Big),                    \\
  c_{d_\perp} & = \frac{1}{2}\Big(\beta_{1 + m_+}\alpha_{1 + m_+} - \sum_{i=1}^{m_+} \beta_i \alpha_i\Big).
\end{align}
Denote the affine map $U\circ \Psi$ by $M_1$.

\subsection{Step 3: Hyperbolic and Orthogonal Rotations}

If $m_+=0$, the positive block is empty and there is no
positive coordinate with which to form a hyperbolic-rotation plane. In this case
we skip the rotations below, set $M_3=M_1$, $t_1=0$,
$t_2=\beta_{1 + m_+}$, and
\[
  \Delta=-t_2^2.
\]
The centered canonical map is then
\[
  \relx{y}_1=-\relx{z}_1^2,\qquad
  \relx{y}_2=t_2\relx{z}_1.
\]
For the rest of this step, assume $m_+\ge 1$.

\begin{enumerate}
  \item \textbf{Orthogonal rotation in $\bR^{m_+}$.}
        Next, apply an orthogonal rotation in $\bR^{m_+}$. This rotation
        concentrates the linear dependence of
        \[
          F(M_1(v))^\top d_\perp
        \]
        on $(v_1,\dots,v_{m_+})$ into the single coordinate $z_1$.
        Write $\beta_+ \doteq (\beta_1,\dots,\beta_{m_+})^\top$ and choose
        $R\in O(m_+)$ such that
        \[
          R\beta_+ = (\beta'_1,0,\dots,0)^\top,
          \qquad
          \beta'_1 \doteq \norm{\beta_+}.
        \]
        Define $P:\bR^n\to\bR^n$ by
        \[
          (P(z))_i = (R^\top z_+)_i \quad\text{for } i\le m_+,
          \qquad
          (P(z))_i = z_i \quad\text{for } i>m_+,
        \]
        where $z_+\doteq(z_1,\dots,z_{m_+})^\top$.
        Since $R$ is orthogonal, $P$ is a bijection on $\bR^n$.
        Reparametrizing with $v=P(z)$ therefore leaves the image unchanged:
        \[
          F(\bR^n)
          =
          \left\{
          F(M_1(P(z)))^\top d\,d
          +
          F(M_1(P(z)))^\top d_\perp\,d_\perp
          :\; z\in\bR^n
          \right\}.
        \]
        Since $R$ is orthogonal, $\|v_+\|^2=\|R^\top z_+\|^2=\|z_+\|^2$,
        so the quadratic function is preserved. The linear term becomes
        \[
          \beta_+^\top v_+=(R\beta_+)^\top z_+=\beta'_1 z_1.
        \]
        A direct substitution gives
        \begin{align*}
          F(M_1(P(z)))^\top d
           & = z_1^2+\sum_{i=2}^{m_+} z_i^2 - z_{1 + m_+}^2 + c_d,       \\
          F(M_1(P(z)))^\top d_\perp
           & = \beta'_1 z_1 + \beta_{1 + m_+} z_{1 + m_+} + c_{d_\perp}.
        \end{align*}
        Denote $M_1\circ P$ by $M_2$.
  \item \textbf{Hyperbolic rotation in the $(z_1,z_{1 + m_+})$-plane.}
        Next, apply a hyperbolic rotation in the $(z_1,z_{1 + m_+})$-plane that
        simplifies the linear coefficients in
        $F(M_2(z))^\top d_\perp$.
        Define $\hat{H}:\bR^n\to\bR^n$ by
        \begin{align*}
          (\hat{H}(\relx{z}))_1         & = \cosh\psi\,\relx{z}_1+\sinh\psi\,\relx{z}_{1 + m_+}, \\
          (\hat{H}(\relx{z}))_{1 + m_+} & = \sinh\psi\,\relx{z}_1+\cosh\psi\,\relx{z}_{1 + m_+},
        \end{align*}
        and $(\hat{H}(\relx{z}))_i=\relx{z}_i$ for $i\notin\{1,1 + m_+\}$, where
        $\psi\in\bR$.  Since the matrix
        \[
          H(\psi)\doteq
          \begin{pmatrix}
            \cosh\psi & \sinh\psi \\
            \sinh\psi & \cosh\psi
          \end{pmatrix}
        \]
        is invertible ($\det H(\psi)=\cosh^2\psi-\sinh^2\psi=1$),
        $\hat{H}$ is a bijection on $\bR^n$.
        Reparametrizing with $z=\hat{H}(\relx{z})$ therefore leaves the image
        unchanged:
        \[
          F(\bR^n)
          =
          \left\{
          F(M_2(\hat{H}(\relx{z})))^\top d\,d
          +
          F(M_2(\hat{H}(\relx{z})))^\top d_\perp\,d_\perp
          :\;\relx{z}\in\bR^n
          \right\}.
        \]
        Setting $J\doteq\operatorname{diag}(1,-1)$, the identity
        $H(\psi)^\top J H(\psi)=J$ holds for every $\psi\in\bR$,
        so the Minkowski form is preserved:
        $z_1^2-z_{1 + m_+}^2=\relx{z}_1^2-\relx{z}_{1 + m_+}^2$.
        The linear part transforms as
        $t\doteq H(\psi)^\top\beta$ with
        $\beta\doteq(\beta'_1,\beta_{1 + m_+})^\top$, giving
        \begin{align*}
          t_1 & = \beta'_1\cosh\psi+\beta_{1 + m_+}\sinh\psi, \\
          t_2 & = \beta'_1\sinh\psi+\beta_{1 + m_+}\cosh\psi.
        \end{align*}
        Because $t^\top Jt=\beta^\top J\beta$, the invariant
        \begin{equation}
          \label{eq:app-invariant}
          \Delta\doteq t_1^2-t_2^2=(\beta'_1)^2-\beta_{1 + m_+}^2
          =\sum_{i=1}^{m_+}\beta_i^2-\beta_{1 + m_+}^2
        \end{equation}
        is independent of $\psi$.
        A direct substitution gives, for any $\psi\in\bR$,
        \begin{align*}
          F(M_2(\hat{H}(\relx{z})))^\top d
           & = \relx{z}_1^2+\sum_{i=2}^{m_+} \relx{z}_i^2
          -\relx{z}_{1 + m_+}^2+c_d,                            \\
          F(M_2(\hat{H}(\relx{z})))^\top d_\perp
           & = t_1\relx{z}_1+t_2\relx{z}_{1 + m_+}+c_{d_\perp}.
        \end{align*}

        Finally, we choose the hyperbolic angle $\psi$ according to the sign of the invariant $\Delta$, so as to eliminate one coefficient whenever possible:
        \begin{itemize}
          \item If $\Delta > 0$, we choose $\psi$ such that $t_2 = 0$. This requires $\tanh \psi = -\beta_{1 + m_+}/\beta'_1$, which is valid since $|\beta_{1 + m_+}/\beta'_1| < 1$.
          \item If $\Delta < 0$, we choose $\psi$ such that $t_1 = 0$. This requires $\tanh \psi = -\beta'_1/\beta_{1 + m_+}$, which is valid since $|\beta'_1/\beta_{1 + m_+}| < 1$.
          \item If $\Delta = 0$, then $\beta'_1=\pm\beta_{1+m_+}$. If this vector is nonzero, no finite hyperbolic rotation can eliminate one coefficient; we take $\psi=0$, so $t_1=\beta'_1$ and $t_2=\beta_{1+m_+}$. If both coefficients are zero, then $t_2=0$ and the candidate direction is discarded by Proposition~\ref{prop:t2-nonzero}.
        \end{itemize}
\end{enumerate}

When $m_+\ge1$, denote $M_2\circ \hat{H}$ by $M_3$. Since $M_3$ is a bijection on $\bR^n$, the image is unchanged.
Recall the orthonormal matrix
\begin{equation}
  \label{eq.app-D-canon}
  D\doteq [\,d\ \ d_\perp\,],
\end{equation}
and let
\begin{equation}
  \label{eq.app-c-canon}
  c\doteq(c_d,c_{d_\perp})^\top.
\end{equation}
Define the centered coordinates $\relx{y}\doteq D^\top y-c$. The centered canonical
map is
\begin{equation}
  \label{eq:app-u-canon}
  (\relx{y}_1,\relx{y}_2)=
  \begin{cases}
    \bigl(-\relx{z}_1^2,\ t_2\relx{z}_1\bigr), & m_+=0,   \\[2mm]
    \bigl(\sum_{i=1}^{m_+} \relx{z}_i^2-\relx{z}_{1 + m_+}^2,\
    t_1\relx{z}_1+t_2\relx{z}_{1 + m_+}\bigr), & m_+\ge1.
  \end{cases}
\end{equation}
As $\relx{z}$ ranges over $\bR^n$, this map defines $\relx{S}_{\JR}$, and
the original joint range is recovered as
\[
  F(\bR^n)=D(\relx{S}_{\JR}+c).
\]

\begin{proposition}
  \label{prop:t2-nonzero}
  Assume that conditions \textup{(b1)--(b3)} and the canonical map
  \eqref{eq:app-u-canon} hold. Condition~\textup{(b4)} holds
  if and only if $\Delta \le 0$ \emph{and} $t_2\ne 0$.
\end{proposition}

\begin{proof}
  Note that $\Delta = t_1^2 - t_2^2$. If $m_+=0$, then $t_1=0$,
  $\Delta=-t_2^2$, and condition~\textup{(b4)} reduces to
  \[
    \relx{z}_1^2=1 \;\Longrightarrow\; t_2\relx{z}_1\neq 0,
  \]
  which is equivalent to $t_2\neq0$. Hence the result holds for $m_+=0$.

  Assume now that $m_+\ge1$. Condition~\textup{(b4)} concerns the homogeneous
  preimage $F_H^{-1}(-d)$, so the affine translation $\Psi$ used to recenter the
  inhomogeneous image is not applied. Let $r$ denote the coordinates obtained
  from the original diagonal coordinates by the same orthogonal and hyperbolic
  rotations, but without the translation; equivalently, write
  $u=U P(\hat{H}(r))$. Since these rotations preserve
  $\sum_{i=1}^{m_+}r_i^2-r_{1+m_+}^2$, condition~\textup{(b4)} is equivalent to
  \begin{equation}
    \label{eq:b4-canon}
    r_{1 + m_+}^2 = 1+\sum_{i=1}^{m_+} r_i^2
    \;\Longrightarrow\;
    t_1r_1+t_2r_{1 + m_+}\neq 0.
  \end{equation}

  \noindent($\Rightarrow$)
  Assume that (b4) holds.
  Taking $r_{1 + m_+}=1$ and $r_i=0$ for $i\neq 1 + m_+$ satisfies
  the left-hand side of \eqref{eq:b4-canon}, giving $t_2\neq 0$.

  It remains to show $\Delta\le 0$.
  Assume that for contradiction $\Delta>0$ (i.e., $t_1^2>t_2^2$).
  Set
  \[
    r_1 = \frac{t_2}{\sqrt{\Delta}},\quad
    r_{1 + m_+} = -\frac{t_1}{\sqrt{\Delta}},\quad
    r_i = 0 \text{ for } 2\le i\le m_+.
  \]
  Then
  \[
    r_{1 + m_+}^2 = \frac{t_1^2}{\Delta}
    = \frac{\Delta+t_2^2}{\Delta}
    = 1+\frac{t_2^2}{\Delta}
    = 1+r_1^2,
  \]
  so the left-hand side of \eqref{eq:b4-canon} holds.  But
  $t_1r_1+t_2r_{1 + m_+}
    = \tfrac{t_1 t_2}{\sqrt{\Delta}}-\tfrac{t_1 t_2}{\sqrt{\Delta}}=0$,
  contradicting \eqref{eq:b4-canon}. Hence $\Delta\le 0$.

  \noindent($\Leftarrow$)
 Assume $\Delta \le 0$ (i.e., $t_1^2 \le t_2^2$) and $t_2\ne 0$.
Let $r$ satisfy $r_{1 + m_+}^2=1+\sum_{i=1}^{m_+}r_i^2$
and suppose for contradiction that $t_1r_1+t_2r_{1 + m_+}=0$.
Since $t_2\ne 0$, we may square and rearrange:
$t_2^2r_{1 + m_+}^2=t_1^2r_1^2\le t_2^2r_1^2$,
so $r_{1 + m_+}^2\le r_1^2$.
But $\sum_{i=1}^{m_+}r_i^2\ge r_1^2$, so
$r_{1 + m_+}^2=1+\sum_{i=1}^{m_+}r_i^2\ge 1+r_1^2>r_1^2$,
a contradiction. Hence \eqref{eq:b4-canon} holds, and therefore condition~\textup{(b4)} is satisfied by the equivalence above.
\end{proof}

\subsection{Classification of Shapes of Joint Ranges}
\label{sec:app-classification}

When conditions \textup{(b1)--(b4)} are satisfied for the candidate direction $d$, the above transformations put $F$ into the canonical form \eqref{eq:app-u-canon}.
Let $\relx{S}_{\JR}$ denote the image of the centered canonical map \eqref{eq:app-u-canon}:
\begin{equation*}
  \relx{S}_{\JR} = \{ (\relx{y}_1, \relx{y}_2) \in \bR^2 : \exists\,\relx{z}\in\bR^n\text{ such that }(\relx{y}_1,\relx{y}_2)\text{ is given by \eqref{eq:app-u-canon}} \}.
\end{equation*}
Then $F(\bR^n) = D(\relx{S}_{\JR}+c)$, where $D$ is the linear map in
\eqref{eq.app-D-canon} and $c$ is the translation vector in
\eqref{eq.app-c-canon}. The offset $c$ is not part of $\relx{S}_{\JR}$; it is
added separately to recover $F(\bR^n)$.

In the centered coordinates, and up to the rotations and translations described
above, the joint range has the following nonconvex shapes.

\subsubsection{Category 1: The Parabolic Nonconvex Case \texorpdfstring{($\Delta < 0$)}{(t1 squared less than t2 squared)}}

When $\Delta < 0$ (i.e., $t_1^2 < t_2^2$), the case $m_+=0$ already has
$t_1=0$; when $m_+\ge1$, a hyperbolic change of variables may be chosen so
that $t_1=0$ and $t_2\neq0$. Since $t_1=0$, the invariant gives
$\Delta = -t_2^2$, hence $t_2^2 = -\Delta > 0$. The centered canonical map
\eqref{eq:app-u-canon} reduces to
\[
  \relx{y}_1 = \sum_{i=1}^{m_+} \relx{z}_i^2 - \relx{z}_{1 + m_+}^2,
  \qquad
  \relx{y}_2 = t_2 \relx{z}_{1 + m_+}.
\]
Substituting $\relx{z}_{1 + m_+} = \relx{y}_2/t_2$ and using $t_2^2 = -\Delta$ gives
\[
  \relx{y}_1 = -\left(\frac{\relx{y}_2}{t_2}\right)^2 + \sum_{i=1}^{m_+} \relx{z}_i^2
  = \frac{\relx{y}_2^2}{\Delta} + \sum_{i=1}^{m_+} \relx{z}_i^2.
\]
\begin{itemize}
  \item \textbf{Sub-case 1a ($m_+ = 0$).} The sum is empty and the range is the parabola
        \[
          \relx{S}_{\JR} = \{ (\relx{y}_1, \relx{y}_2) \in \bR^2 : \relx{y}_1 = \relx{y}_2^2/\Delta\}.
        \]
        Since $\Delta<0$, this is a left-opening parabola ($\relx{y}_1 \le 0$).

  \item \textbf{Sub-case 1b ($m_+ \ge 1$).} Here the sum $\sum \relx{z}_i^2$ can take any value in $[0,\infty)$, so the range is the solid parabolic region
        \[
          \relx{S}_{\JR} = \{ (\relx{y}_1, \relx{y}_2) \in \bR^2 : \relx{y}_1 \ge \relx{y}_2^2/\Delta \}.
        \]
\end{itemize}

\subsubsection{Category 2: The Punctured Nonconvex Case \texorpdfstring{($\Delta = 0$)}{(t1 squared equals t2 squared)}}

When $\Delta = 0$ (i.e., $t_1^2 = t_2^2$) and condition~\textup{(b4)} holds,
Proposition~\ref{prop:t2-nonzero} gives $t_2\ne0$, so there is a sign
$\sigma\in\{-1,1\}$ such that
\[
  t_1=\sigma t_2.
\]
This implies $m_+ \ge 1$, since $m_+=0$ would give $\Delta=-t_2^2<0$.

The canonical map \eqref{eq:app-u-canon} then gives
\[
  \relx{y}_1=\sum_{i=1}^{m_+} \relx{z}_i^2-\relx{z}_{1 + m_+}^2
  =(\relx{z}_1-\sigma\relx{z}_{1 + m_+})
  (\relx{z}_1+\sigma\relx{z}_{1 + m_+})
  +\sum_{i=2}^{m_+} \relx{z}_i^2,
\]
\[
  \relx{y}_2=t_2(\sigma\relx{z}_1+\relx{z}_{1 + m_+})
  =\sigma t_2(\relx{z}_1+\sigma\relx{z}_{1 + m_+}).
\]
Set
\[
  s\doteq \relx{z}_1+\sigma\relx{z}_{1 + m_+}
  =\frac{\relx{y}_2}{\sigma t_2},
  \qquad
  r\doteq \relx{z}_1-\sigma\relx{z}_{1 + m_+}.
\]
Hence
\[
  \relx{y}_1=r\frac{\relx{y}_2}{\sigma t_2}+\sum_{i=2}^{m_+} \relx{z}_i^2.
\]
Now fix $\relx{y}_2$:
\begin{itemize}
  \item if $\relx{y}_2\neq 0$, then the coefficient $\relx{y}_2/(\sigma t_2)$ is nonzero and $r$ is free, so $\relx{y}_1$ can be any real number;
  \item if $\relx{y}_2=0$, then $s=0$, so the mixed term vanishes and
        \[
          \relx{y}_1=\sum_{i=2}^{m_+} \relx{z}_i^2\ge 0.
        \]
\end{itemize}
Therefore the only missing points lie on the axis $\relx{y}_2=0$, leading to two punctured shapes:
\begin{itemize}
  \item \textbf{Sub-case 2a ($m_+ = 1$).} The sum is empty, so if $\relx{y}_2=0$ then $\relx{y}_1=0$ and
        \[
          \relx{S}_{\JR} = \bR^2 \setminus \{ (\relx{y}_1, 0) \in \bR^2 : \relx{y}_1 \neq 0 \}.
        \]

  \item \textbf{Sub-case 2b ($m_+ \ge 2$).} The sum can take any value in $[0,\infty)$, so if $\relx{y}_2=0$ then $\relx{y}_1 \ge 0$ and
        \[
          \relx{S}_{\JR} = \bR^2 \setminus \{ (\relx{y}_1, 0) \in \bR^2 : \relx{y}_1 < 0 \}.
        \]
\end{itemize}

\subsection{Summary of the Classification}

\begin{theorem}
  \label{thm:classification}
  Let
  \[
    F(x) = (x^\top \Theta_1 x + \theta_1^\top x,\ x^\top \Theta_2 x + \theta_2^\top x).
  \]
  Let $\cN_0 \doteq \ker \Theta_1 \cap \ker \Theta_2$. Then:
  \begin{itemize}
    \item If $\Theta_1=\Theta_2=0$, then $F(\bR^n)$ is convex.
    \item If $\Theta_1$ and $\Theta_2$ are linearly independent, then $F(\bR^n)$ is convex.
    \item If $\{\theta_1,\theta_2\} \not\subseteq \cN_0^\perp$, then $F(\bR^n)$ is convex.
    \item Otherwise, there exists a nonzero direction $d=(d_1,d_2)$ (unique up to scaling) such that $d_2 \Theta_1 = d_1 \Theta_2$. Normalize $d$ so that $\|d\|_2=1$; the two candidate directions are $d$ and $-d$. For each normalized candidate direction $\hat d$, set $\hat d_\perp=(-\hat d_2,\hat d_1)$ and $D_{\hat d}\doteq[\,\hat d\ \hat d_\perp\,]$. Let $\bar Q_{\hat d}\doteq \hat d_1\Theta_1+\hat d_2\Theta_2$, and let $(m_+,m_-,m_0)$ be the inertia of $\bar Q_{\hat d}$ on $\cN_0^\perp$ (so $m_0=0$ in exact arithmetic).
          \begin{itemize}
            \item If $m_-\neq 1$, then this direction cannot produce a nonconvex shape.
            \item Otherwise, the canonical reduction gives $m_+$, $\Delta$, $t_1$, $t_2$,
                  the offset $c$, and the centered shape map \eqref{eq:app-u-canon}.
                  \begin{itemize}
                    \item If $\Delta>0$ or $t_2=0$, then condition~\textup{(b4)}
                          fails for this candidate direction, so this direction is discarded.
                    \item If $\Delta\le0$ and $t_2\ne0$, then condition~\textup{(b4)}
                          holds and $F(\bR^n)$ is nonconvex. In centered canonical
                          coordinates, $\relx{S}_{\JR}$ has one of the following
                          forms; the original range is recovered as
                          $F(\bR^n)=D_{\hat d}(\relx{S}_{\JR}+c)$.
                          \begin{itemize}
                            \item If $\Delta < 0$ and $m_+ = 0$, then
                                  \[
                                    \relx{S}_{\JR}
                                    = \{(\relx{y}_1,\relx{y}_2)\in\bR^2 :
                                    \relx{y}_1 = \relx{y}_2^2/\Delta\},
                                  \]
                                  a one-dimensional parabola.

                            \item If $\Delta < 0$ and $m_+ \ge 1$, then
                                  \[
                                    \relx{S}_{\JR}
                                    = \{(\relx{y}_1,\relx{y}_2)\in\bR^2 :
                                    \relx{y}_1 \ge \relx{y}_2^2/\Delta\},
                                  \]
                                  a solid parabolic region.

                            \item If $\Delta = 0$ and $m_+ = 1$, then
                                  \[
                                    \relx{S}_{\JR}
                                    = \bR^2 \setminus
                                    \{(\relx{y}_1,0)\in\bR^2 : \relx{y}_1 \neq 0\},
                                  \]
                                  the plane with the horizontal axis removed except for the origin.

                            \item If $\Delta = 0$ and $m_+ \ge 2$, then
                                  \[
                                    \relx{S}_{\JR}
                                    = \bR^2 \setminus
                                    \{(\relx{y}_1,0)\in\bR^2 : \relx{y}_1 < 0\},
                                  \]
                                  a plane missing a ray.
                          \end{itemize}
                  \end{itemize}
          \end{itemize}
          If neither candidate direction yields a nonconvex shape, then $F(\bR^n)$ is convex.
  \end{itemize}
\end{theorem}
\begin{proof}
  The first three cases follow from the affine case, \Cref{cor:LD-necessary}, and condition~\textup{(C1)} in the global convexity characterization \Cref{thm:fbop-global-convexity}. In the remaining linearly dependent case, \Cref{prop:Q-cases} reduces any nonconvex candidate direction to the case $m_-=1$, \Cref{prop:b1-active} removes the common-kernel variables from the linear part, and \Cref{prop:t2-nonzero} shows that the fixed-direction nonconvexity condition is equivalent to $\Delta\le0$ and $t_2\ne0$. The explicit shape descriptions are exactly the centered canonical images derived in Categories~1 and~2 above. If neither normalized candidate direction satisfies these conditions, then no direction satisfies the nonconvexity criterion in \Cref{char.nonconvexity}, and therefore $F(\bR^n)$ is convex.
\end{proof}

The final sentence of \Cref{thm:classification} should be read candidate-by-candidate.  If, after the canonical reduction for a fixed normalized direction $d$, one obtains $\Delta>0$ or $t_2=0$, then Proposition~\ref{prop:t2-nonzero} shows that condition~\textup{(b4)} fails for this candidate direction.  Hence this candidate direction cannot certify nonconvexity, but the opposite candidate direction $-d$ must still be checked before concluding convexity.

For example, consider the centered canonical map
\[
  \relx{y}_1=z_1^2-z_2^2,\qquad \relx{y}_2=z_1.
\]
Here $t_1=1$, $t_2=0$, and $\Delta=1>0$.  On the hyperbola defining $F_H^{-1}(-d)$, namely $z_2^2=1+z_1^2$, one may take $z_1=0$ and $z_2=\pm1$, which makes the $d_\perp$-component of the linear part equal to zero.  Thus the linear part is parallel to $d$, condition~\textup{(b4)} fails, and this candidate direction should be discarded.

As a degenerate case, if
\[
  \relx{y}_1=z_1^2-z_2^2,\qquad \relx{y}_2=0,
\]
then the image is simply
\[
  \bR\times\{0\},
\]
which is convex.  Here $t_1=t_2=0$, so $t_2=0$ and $\Delta=0$, and no nonconvex shape appears from this candidate direction.

\begin{example}
  \label{ex:parabolic}
  To illustrate the classification of the joint range on a higher-dimensional instance, consider a pair of quadratic functions defined on $\bR^4$:
  \begin{align*}
    f_1(x) & = \frac{1}{\sqrt{2}} \left(x_1^2 + x_2^2 + x_3^2 - x_4^2 - x_4\right), \\
    f_2(x) & = \frac{1}{\sqrt{2}} \left(x_1^2 + x_2^2 + x_3^2 - x_4^2 + x_4\right).
  \end{align*}
  Here, the quadratic matrices and linear vectors are:
  \[
    \Theta_1 = \Theta_2 = \frac{1}{\sqrt{2}} \begin{pmatrix} 1 & 0 & 0 & 0 \\ 0 & 1 & 0 & 0 \\ 0 & 0 & 1 & 0 \\ 0 & 0 & 0 & -1 \end{pmatrix}, \quad
    \theta_1 = \frac{1}{\sqrt{2}} \begin{pmatrix} 0 \\ 0 \\ 0 \\ -1 \end{pmatrix}, \quad
    \theta_2 = \frac{1}{\sqrt{2}} \begin{pmatrix} 0 \\ 0 \\ 0 \\ 1 \end{pmatrix}.
  \]
  Algorithm~\ref{alg:joint-range} proceeds as follows:
  \begin{enumerate}
    \item $\Theta_1$ and $\Theta_2$ are linearly dependent, since $1 \cdot \Theta_1 = 1 \cdot \Theta_2$. The normalized direction satisfying $d_2 \Theta_1 = d_1 \Theta_2$ is $d = (\frac{1}{\sqrt{2}}, \frac{1}{\sqrt{2}})$.
    \item The common kernel is $\cN_0 = \ker(\Theta_1) \cap \ker(\Theta_2) = \{0\}$, so $\cN_0^\perp = \bR^4$, satisfying the inclusion condition.
    \item For the candidate direction $\hat{d} = (\frac{1}{\sqrt{2}}, \frac{1}{\sqrt{2}})$, we form $\bar Q_{\hat d} = \hat{d}_1 \Theta_1 + \hat{d}_2 \Theta_2 = \operatorname{diag}(1, 1, 1, -1)$. Its eigenvalues are $\{1, 1, 1, -1\}$, giving an inertia of $m_+ = 3$ and $m_- = 1$.
    \item Since $m_- = 1$, the canonical reduction proceeds. The orthogonal direction is $\hat{d}_\perp = (-\frac{1}{\sqrt{2}}, \frac{1}{\sqrt{2}})$. We project the linear parts onto the $\hat d_\perp$-component to find the canonical coefficients:
          \begin{align*}
            \hat d_\perp^\top(\theta_1^\top x,\theta_2^\top x)^\top
             & = -\frac{1}{\sqrt{2}}\left(-\frac{1}{\sqrt{2}}x_4\right)
            +\frac{1}{\sqrt{2}}\left(\frac{1}{\sqrt{2}}x_4\right)
            = x_4.
          \end{align*}
          Hence the negative-coordinate coefficient is $t_2=1$ and there is no positive-block linear coefficient, so $t_1=0$.
          This yields the geometric invariant $\Delta = t_1^2 - t_2^2 = 0 - 1 = -1$.
    \item Because $\Delta < 0$ and $m_+ = 3 \ge 1$, the algorithm classifies $F(\bR^4)$ as \textsc{Nonconvex}, mapping to the canonical solid parabolic region $\relx{S}_{\JR} = \{ (\relx{y}_1, \relx{y}_2) : \relx{y}_1 \ge -\relx{y}_2^2 \}$.
  \end{enumerate}
  The final joint range is the rotated solid parabolic region $F(\bR^4) = D(\relx{S}_{\JR})$, where $D = \left[ \hat{d}, \hat{d}_\perp \right] = \frac{1}{\sqrt{2}} \left(\begin{smallmatrix} 1 & -1 \\ 1 & 1 \end{smallmatrix}\right)$.
\end{example}

\section{Experimental Instance Generation}
\label{sec:app-experiment-generation}

This appendix describes how the random instances in \Cref{sec:experimental} are generated. The purpose of the generator is not to model a particular application. Instead, it creates controlled examples whose two-dimensional projected hull falls into one of the seven geometric cases in Table~\ref{tab:mccormick-ratio}. For every dimension $n\in\{3,4,5,6,7,8\}$ and every case, we keep ten accepted random instances.

We use two kinds of uniform sampling. The notation $\mathrm{Uniform}[a,b]$ means the continuous uniform distribution on the interval $[a,b]$, whereas $\mathrm{Uniform}\{a,\ldots,b\}$ means the discrete uniform distribution on the listed integers. We also use random orthogonal matrices. When we write $R\sim\mathrm{Haar}(O(n))$, this means that $R$ is sampled uniformly from the orthogonal group, so no coordinate direction is preferred.

\medskip
\noindent\textbf{Nonconvex parabolic cases.}
Cases~1--6 are generated from the canonical nonconvex parabolic form in Appendix~\ref{sec:app-canonical-transformation}. Each trial begins by drawing
\[
  \Delta\sim \mathrm{Uniform}[-2,-0.5],
  \qquad
  t_2=\sqrt{-\Delta}.
\]
The value of $m_+$ determines whether the canonical joint range is only the parabolic boundary or the solid parabolic region. We set $m_+=0$ in Cases~1, 3, and 5. In Cases~2, 4, and 6, we draw $m_+$ uniformly from $\{1,\ldots,n-1\}$. Thus the odd-numbered nonconvex cases use the boundary range $\bd(C_{\Par})$, and the even-numbered nonconvex cases use the solid range $\clos(C_{\Par}^c)$.

The canonical quadratic part is
\[
  \Lambda_{m_+}=\operatorname{diag}(\underbrace{1,\ldots,1}_{m_+},-1,0,\ldots,0),
\]
and the canonical seed map is
\[
  \relx{F}_{\Delta,m_+}(\relx{z})
  =
  \left(
  \relx{z}^{\top}\Lambda_{m_+}\relx{z},
  t_2\relx{z}_{1+m_+}
  \right).
\]
The sampled quadratic functions are obtained by reversing the canonical transformation in Appendix~\ref{sec:app-canonical-transformation}. The reverse construction is as follows.

\begin{enumerate}[label=(R\arabic*)]
  \item \textbf{Start from the canonical form after Step~3.}
        The map $\relx{F}_{\Delta,m_+}$ is already a centered canonical seed of the form \eqref{eq:app-u-canon}; no input translation is added at this stage.

  \item \textbf{Choose a pre-reduction active linear vector.}
        Step~3 of Appendix~\ref{sec:app-canonical-transformation} uses an
        orthogonal rotation in the positive block and a hyperbolic rotation in one
        positive/negative active plane to simplify the linear part. For
        solid-range table cases, draw $\eta\sim\mathrm{Uniform}[0,1]$ and set the
        pre-reduction linear vector in that active plane to
        \[
          \bar\theta
          =
          -t_2\sinh(\eta)e_{m_+}
          +t_2\cosh(\eta)e_{1+m_+}.
        \]
        Equivalently, if
        \[
          H(\eta)=
          \begin{pmatrix}
            \cosh\eta & \sinh\eta \\
            \sinh\eta & \cosh\eta
          \end{pmatrix},
        \]
        then the two active coefficients of $\bar\theta$ are
        $H(-\eta)(0,t_2)^\top$. Applying the forward hyperbolic rotation with
        angle $\eta$ recovers the canonical vector $t_2e_{1+m_+}$, and the
        quadratic matrix $\Lambda_{m_+}$ is unchanged. For boundary-range table
        cases, $m_+=0$, so there is no positive coordinate to pair with the
        negative coordinate and $\bar\theta=t_2e_1$.

  \item \textbf{Randomize the inverse of Steps~1 and 2.}
        Steps~1 and 2 diagonalize the quadratic part, introduce active coordinates, and record the translation $c$ in \eqref{eq.app-c-canon}. In the experiments the seed is centered before the sampled output mixing. After the final mixing by $G$, Algorithm~\ref{alg:joint-range} may recover a nonzero canonical offset $c$; the acceptance test and hull construction use this recovered offset. We reverse the coordinate part by drawing $R\sim\mathrm{Haar}(O(n))$ and using the change of variables $\relx{z}=R^\top x$. This gives the two building blocks
        \[
          x^\top R\Lambda_{m_+}R^\top x,
          \qquad
          (R\bar\theta)^\top x .
        \]
        This random rotation makes the sampled instance look generic in the original $x$ variables.

  \item \textbf{Choose the sampled output mixing matrix.}
        The forward canonical transformation maps centered canonical coordinates back to the original image space through the orthonormal matrix $D$ in \eqref{eq.app-D-canon}. In the generator, however, the case-dependent output mixing is a sampled matrix $G$, not this canonical $D$. Algorithm~\ref{alg:joint-range} later recovers its own canonical matrix $D$ from the generated instance. The final coefficients of $F=(f_1,f_2)$ are
        \[
          \Theta_i=G_{i1}R\Lambda_{m_+}R^\top,
          \qquad
          \theta_i=G_{i2}R\bar\theta,
          \qquad i=1,2.
        \]
\end{enumerate}

The sampled matrix $G\in\bR^{2\times2}$ is chosen differently for the three geometric configurations. For Cases~1 and 2, the two output rows are placed near the recession direction of the parabolic bowl. More precisely, with $u(\alpha)=(\cos\alpha,\sin\alpha)$, we draw
\[
  \alpha\sim\mathrm{Uniform}[\pi-0.3,\pi+0.3],
  \qquad
  \beta\sim\mathrm{Uniform}[0.5,1.0],
\]
and use rows $u(\alpha-\beta)^\top$ and $u(\alpha+\beta)^\top$. This makes containment configurations occur more often. For Cases~3 and 4, we use a random planar orthogonal matrix $Q\sim\mathrm{Haar}(O(2))$, which gives an unbiased orientation and is suited for the chord-truncated configurations. For Cases~5 and 6, one row is fixed in the recession direction $(-1,0)$, and the other is an axis direction with random sign; the two rows are also randomly permuted. This biases the sampler toward the ray-truncated configurations.

Finally, the two base inequalities are shifted by independent offsets
\[
  \phi_1,\phi_2\sim \mathrm{Uniform}[-2,2],
  \qquad
  K_{\JR}=\{y\in\bR^2:\phi_1+y_1\ge0,\ \phi_2+y_2\ge0\}.
\]
The shifts move the cone $K_{\JR}$ relative to the parabolic joint range.
After each reverse-canonical draw, Algorithm~\ref{alg:joint-range} is applied
to the resulting pair of quadratic functions. The draw is retained only if the
recovered data are nonconvex, have $\Delta<0$, and the transformed cone
$\relx{K}_{\JR}$ satisfies the geometric predicate for the target table case:
containment of the parabolic bowl, two finite boundary-ray intersections with
the cone apex inside the bowl, or one recession ray with the cone apex inside
the bowl.

\medskip
\noindent\textbf{Convex Spectrahedral Shadow.}
Case~7 is generated separately because its joint range is convex and is evaluated through the Shor SDP relaxation. We choose a random orthogonal basis $R_{\rm psd}$ by orthogonalizing a Gaussian matrix. In that basis, the first quadratic function has eigenvalues evenly spaced from $1$ to $2$, and the second has eigenvalues evenly spaced from $2$ down to $1$. Equivalently, for $j=1,\ldots,n$ let
\[
  \lambda^{(1)}_j = 1 + \frac{j-1}{n-1},
  \qquad
  \lambda^{(2)}_j = 2 - \frac{j-1}{n-1},
\]
and set
\[
  \Theta_1 = R_{\rm psd}\operatorname{diag}(\lambda^{(1)}_1,\ldots,\lambda^{(1)}_n)R_{\rm psd}^{\top},
  \qquad
  \Theta_2 = R_{\rm psd}\operatorname{diag}(\lambda^{(2)}_1,\ldots,\lambda^{(2)}_n)R_{\rm psd}^{\top}.
\]
The linear terms are small Gaussian perturbations,
\[
  \theta_i\sim 0.1\,\mathcal{N}(0,I_n),\qquad i=1,2,
\]
and the offsets are fixed at $\phi_1=\phi_2=0$.

\medskip
\noindent\textbf{Numerical approximation.}
For every accepted instance, the projection of the first-level box RLT polytope is approximated by support-function sampling: we use 360 equally spaced directions in the $(y_1,y_2)$ plane and solve one LP over $\mathcal R_{\rm RLT}$ in each direction. The resulting projected polygon is then clipped by the two base inequalities defining $K_{\JR}$. In Case~7, the spectrahedral shadow is approximated by 180 support SDP solves in which the same cone constraints are imposed directly, using SCS with tolerance $10^{-5}$ and at most 20000 iterations. The resulting support points are converted into polygons, and Shapely is used to compute intersections and areas. The table reports the shifted geometric mean of the resulting ratios over the ten accepted instances for each dimension and case.

\end{document}